\documentclass[12pt]{amsart}

\usepackage{amsmath}
\usepackage{amssymb}
\usepackage{array}
\usepackage{enumitem}
\usepackage{url}
\usepackage{verbatim} 
\usepackage{xcolor}
\usepackage{xypic}

\theoremstyle{plain}
\newtheorem{theorem}{Theorem}[section]
\newtheorem{lemma}[theorem]{Lemma}

\newtheorem{corollary}[theorem]{Corollary}

\newtheorem{definition}[theorem]{Definition}
\theoremstyle{remark}
\newtheorem{remark}[theorem]{Remark}
\newtheorem{example}[theorem]{Example}

\numberwithin{equation}{section}

\newcommand{\D}{D}
\newcommand{\bA}{\mathbb{A}}

\newcommand{\bB}{\mathbb{B}}
\newcommand{\B}{\mathbb{B}}

\newcommand{\K}{\mathbb{K}}
\newcommand{\bk}{\mathbf{k}}
\newcommand{\bK}{\mathbb{K}}

\newcommand{\R}{\mathbb{R}}

\newcommand{\C}{\mathbb{C}}
\newcommand{\N}{\mathbb{N}}

\newcommand{\bW}{\mathbb{W}}

\newcommand{\cM}{{\mathcal M}}
\newcommand{\Man}{{\mathcal Man}}
\newcommand{\Bun}{{\mathcal Bun}}
\newcommand{\SBun}{{\mathcal SBun}}
\newcommand{\cN}{\mathcal{N}}
\newcommand{\CC}{\mathcal{C}}

\newcommand{\aaa}{\boldsymbol{a}}
\newcommand{\sss}{\boldsymbol{s}}

\newcommand{\ttt}{\boldsymbol{t}}
\newcommand{\vvv}{\boldsymbol{v}}

\newcommand{\uuu}{\boldsymbol{u}}

\newcommand{\ooo}{\boldsymbol{0}}
\newcommand{\Alpha}{{\boldsymbol{\alpha}}}

\newcommand{\Diff}{\mathrm{Diff}}
\newcommand{\Gl}{\mathrm{Gl}}

\newcommand{\Aut}{\mathrm{Aut}}

\newcommand{\Pol}{\mathrm{Pol}}
\newcommand{\GPol}{\mathrm{GP}}
\newcommand{\GL}{\mathrm{GL}}

\newcommand{\Tay}{\mathrm{Tay}}

\newcommand{\id}{\mathrm{id}}

\newcommand{\pr}{\mathrm{pr}}

\newcommand{\SJ}{{\mathrm{J}}}

\newcommand{\eps}{\varepsilon}

\newcommand{\Bigsetof}[2]{\begin{Bmatrix} #1 \,\Big|\, #2 \end{Bmatrix}}

\newcommand{\inv}{^{-1}}
\newcommand{\msk}{\medskip}
\newcommand{\ssk}{\smallskip}
\newcommand{\nin}{\noindent}

\newcommand{\la}{\langle}
\newcommand{\ra}{\rangle}
\newcommand{\tb}{\textcolor{blue}}

\newcommand{\ts}{\textcolor{orange}}

\DeclareSymbolFont{boldsymbols}{OMS}{cmsy}{b}{n}

\begin{document}

\title{A general construction of Weil functors}

\author{Wolfgang Bertram, Arnaud Souvay}

\address{Institut \'{E}lie Cartan Nancy \\
Nancy-Universit\'{e}, CNRS, INRIA \\
Boulevard des Aiguillettes, B.P. 239 \\
F-54506 Vand\oe{}uvre-l\`{e}s-Nancy, France}

\email{\url{Wolfgang.Bertram@iecn.u-nancy.fr}\\
\url{Arnaud.Souvay@iecn.u-nancy.fr}}



\begin{abstract} 
We construct the Weil functor $T^\bA$ corresponding to a general Weil algebra
$\bA = \K \oplus \cN$: this is a functor from the category of manifolds over a general
topological base field or ring $\K$ (of arbitrary characteristic) to the category of 
manifolds over $\bA$. This result simultaneously generalizes  results known for ordinary,
real manifolds (cf.\ \cite{KMS}), and results obtained in \cite{Be08} for the case
of the higher order tangent functors ($\bA = T^k \K$) and in \cite{Be10} for the case
of jet rings ($\bA = \K[X]/(X^{k+1})$). 
We investigate some algebraic aspects of these general Weil functors
(``$K$-theory of Weil functors'', action of the ``Galois group'' $\Aut_\K(\bA)$),
which will be of importance for subsequent applications to general 
differential geometry.
\end{abstract}

\maketitle

\section{Introduction}

The topic of the present work is the {\em construction and investigation of general Weil functors},  where
the term  ``general'' means: in arbitrary (finite or infinite) dimension, and over general topological
base fields or rings.  
Compared to the (quite vast) existing literature on Weil functors (see, e.g.,
\cite{KMS, Kolar, K00, KM04}), this adds two novel viewpoints:
on the one hand, extension of the theory to a very general context, including, for instance, base fields
of {\em positive characteristic}, and on the other hand, introduction of the point of view of
{\em scalar extension}, well-known in algebraic geometry, into the context of differential geometry.
This aspect is new, even in the context of usual, finite-dimensional real manifolds.
We start  by explaining this item. 

\msk
A quite elementary approach to
differential calculus and -geometry over general base fields or -rings $\K$ has been defined and studied
in \cite{BGN04, Be08}; see \cite{Be11} for an elementary exposition.
The term ``smooth'' always refers to the concept  explained there, and which is called ``cubic smooth''
in \cite{Be10}.  
The base ring $\K$ is a commutative unital topological ring such that  $\K^\times$, the unit group, is
open dense in $\K$, and the 
inversion map $\K^\times \to \K$, $t \mapsto t^{-1}$ is continuous.
For convenience,   the reader may assume that  $\K$  is a topological $\bk$-algebra over some
topological field $\bk$, where  $\bk$ is  his or her favorite field, for instance $\bk = \R$, 
and one may think of $\K$ as  $\R \oplus j \R$ with, for instance,
$j^2 = -1$ ($\K=\C$), or $j^2 = 1$ (the ``para-complex numbers'') or $j^2 = 0$ (``dual numbers'').
In our setting, the analog of the ``classical'' {\em Weil algebras}, as defined, e.g., in \cite{KMS}, is as follows:

\begin{definition}
\label{WeilalgebraDef}
A {\em Weil $\K$-algebra} is a  commutative and associative  $\K$-algebra $\bA$, with unit $1$, of the
form $\bA = \K \oplus \cN_\bA$, where $\cN = \cN_\bA$ is a nilpotent ideal. 
 We assume,
moreover, $\cN$ to be free and finite-dimensional over $\K$. We equip $\bA$ with
 the product topology on $\cN \cong \K^n$ with respect to
some (and hence any) $\K$-basis. 
\end{definition}

As is easily seen (Lemma \ref{WeilalgebraLemma}), $\bA$ is then again of the same kind as $\K$,
hence it is selectable as a new base ring. Since an interesting Weil algebra $\bA$ is never a field,
this explains why we work with base {\em rings}, instead of fields.
Our main results may now be summarized as follows
(see Theorems 
\ref{WeilFunctorTheoremExistence} and \ref{WeilFunctorTheoremBundle}
for details):

\begin{theorem}\label{SummaryTheorem}
Assume $\bA= \K \oplus \cN$ is a Weil $\K$-algebra.
Then, to any smooth $\K$-manifold $M$, one can associate a smooth  manifold $T^\bA M$ such that:
\begin{enumerate}
\item
the construction is functorial and compatible with cartesian products,
\item
$T^\bA M$ is a smooth manifold over $\bA$ (hence also over $\K$), and for any $\K$-smooth map
$f:M \to N$, the corresponding map $T^\bA f:T^\bA M \to T^\bA N$ is smooth over $\bA$,
\item
the manifold $T^\bA M$ is a bundle over the base $M$, and the bundle chart changes in $M$
are polynomial in fibers
(we call this a {\em polynomial bundle}, cf.\ Definition  \ref{ManifoldDefinition}),
\item
if $M$ is an open submanifold $U$ of a topological $\K$-module $V$, then $T^\bA U$ can be 
identified with  the inverse image
of $U$ under the canonical map $V_\bA \to V$, where $V_\bA = V \otimes_\K \bA$ is the usual
scalar extension of $V$; if, in this context, $f:U \to W$ is a {\em polynomial} map, then
$T^\bA f$ coincides with the algebraic scalar extension $f_\bA : V_\bA \to W_\bA$ of $f$.
\end{enumerate}
The {\em Weil functor} $T^\bA$ is uniquely determined by these properties.
\end{theorem}

The {\em Weil bundles} $T^\bA M$ are far-reaching generalizations of the tangent bundle $TM$,
which arises in the special case of the ``dual numbers over $\K$'', $\bA = T \K = \K\oplus \eps \K$ ($\eps^2 = 0$). 
The theorem shows that the structure of the Weil bundle $T^\bA M$ is encoded in the ring 
structure of
$\bA$ in a much stronger form than in the ``classical'' theory (as developed, e.g., in \cite{KMS}):
 the manifold $T^\bA M$ plays in all respects the
r\^ole of a {\em scalar extension} of $M$,  and hence $T^\bA$ can be interpreted
as a {\em functor of scalar extension}, and we could write
$M_\bA := T^\bA M$ -- an  interpretation that is certainly 
 very common for mathematicians used to algebraic geometry, but rather 
  unusual for someone used to classical differential geometry; in this respect, our results
are certainly closer to the original ideas of Andr\'e Weil (\cite{Weil}) than much of the existing literature.  
In subsequent work we will exploit this link between the ``algebraic'' and the ``geometric'' viewpoint to
investigate features of differential geometry, most notably, bundles, connections, and notions of curvature. 

\msk
The algebraic point of view naturally leads to emphasize in differential geometry certain aspects  well-known  from the algebraic theory.
First, Weil algebras and -bundles form a sort of ``$K$-theory'' with respect to the operations

\begin{enumerate}
\item tensor product: $\bA \otimes_\K \bB \cong  \K \oplus (\cN_\bA \oplus \cN_\bB \oplus
\cN_\bA \otimes \cN_\bB)$,
\item Whitney sum: 
$\bA \oplus_\K \bB := (\bA \otimes_\K \bB) / (\cN_\bA \otimes_\K \cN_\B) \cong
\K \oplus \cN_\bA \oplus  \cN_\B$.
\end{enumerate}

\nin
Whereas (2) corresponds exactly to the Whitney sum of the corresponding bundles 
over $M$, one has to be a little bit careful with the bundle interpretation of (1)
(Theorem \ref{KTheoryTheorem}):
Weil bundles are in general {\em not} vector bundles, hence there is no plain notion of
``tensor product''; in fact, (1) rather corresponds to {\em composition of Weil functors}:
$$
T^{\bA \otimes \bB} M \cong  T^{\bB} (T^{\bA} M) .
$$
Second, following the model of Galois theory, for understanding the structure of
the Weil bundle $T^\bA M$ over $M$, it is important to study 
the {\em automorphism group} $\Aut_\K(\bA)$. Indeed, by functoriality, 
any automorphism $\Phi$ of $\bA$ induces canonically
a diffeomorphism $\Phi_M : T^\bA M \to T^\bA M$ which commutes with all $\bA$-tangent maps
$T^\bA f$, for any $f$ belonging to  the $\K$-diffeomorphism group of $M$, $\Diff_\K(M)$.
Thus we have two commuting group actions on $T^\bA M$: one of $\Aut_\K(\bA)$ and 
the other of $\Diff_\K(M)$. As often in group theory, we get a better understanding of one group
action by knowing another group action commuting with it. 
Here, an important special case is the one of a {\em graded Weil algebra} (Chapter 
\ref{sec:Endom}; in \cite{KM04} the term
{\em homogeneous Weil algebra} is used):  in this case, there exist ``one-parameter
subgroups'' of automorphisms, and the Weil algebra carries as new structure  a
``composition like product'', turning it into a (in general
non-commutative) near-ring, similar to the near-ring of
formal power series (Theorem \ref{GradedTheorem}). 

\msk
In \cite{Be08} and \cite{Be10}, two prominent cases of Weil functors have been studied, and
the present work generalizes these results: 
the {\em higher order tangent functors} $T^k$, corresponding to the 
{\em iterated tangent rings},   $T^{k+1}\K := T(T^k\K)$,  
and the {\em jet functors} $\SJ^k$, corresponding to the {\em (holonomic) 
jet rings} $\SJ^k \K := \K[X]/(X^{k+1})$.
Since both cases play a key r\^ole for the proof of our general result, we recall (and
slightly extend) the results for these cases (see Appendices A and B). 
In particular, we describe in some detail the {\em canonical $\K^\times$-action},
which appears  already in the framework of difference calculus, and which gives
rise to the natural grading of these Weil algebras.

\msk
The core part for the proof of Theorem \ref{SummaryTheorem} 
 is a careful investigation of the 
relation between two foundational concepts, namely those of {\em jet}, and of {\em Taylor expansion}.
As is well-known in the classical setting (see, e.g., \cite{Re83}), 
both concepts are essentially equivalent, but the jet-concept
is of an ``invariant'' or ``geometric'' nature, hence makes sense independently of a chart,
whereas Taylor expansions can be written only with respect to  a chart, hence are not
of ``invariant'' nature. This is reflected by a slight difference in their behaviour with respect to
composition of maps:
for jets we have the ``plainly functorial'' composition rule 
\begin{equation}\label{JetRule}
\SJ^k_x (g \circ f) =\SJ^k_{f(x)}g \circ \SJ^k_x f \, ,
\end{equation}
whereas for Taylor polynomials (which we take here without constant term) we 
have truncated composition of polynomials:
\begin{equation}\label{TaylorRule}
\Tay^k_x (g \circ f) =\bigl( \Tay^k_{f(x)}g \circ \Tay^k_x f\bigr) \mod(\deg > k) \, .
\end{equation}
This lack of functoriality is compensated by the advantage that, like every polynomial, the
Taylor polynomial always admits
algebraic scalar extensions, so that the $k$-th order Taylor polynomial
$P = \Tay_x^k f:V \to W$ of a $\K$-smooth function $f:V\supset U \to W$ at $x\in U$ extends  naturally to
a polynomial $P_\bA:V \otimes_\K \bA \to W \otimes_\K \bA$.
Our general construction of Weil functors combines both extension procedures:
in a first step, based on differential calculus reviewed in Appendices A and B,  we construct  the jet functor $\SJ^k$,
which associates to each
smooth map $f$  its (``simplicial'') $k$-jet $\SJ^k f$. 
Using this invariant object, we
 {\em define} in a second step the Taylor polynomial $\Tay_x^k f$ at $x$ by a 
``chart dependent'' construction (Section \ref{Section:Taylor}), and then we  prove the 
transformation rule (\ref{TaylorRule}) (Theorem \ref{TaylorTheorem}). 
Note that, in the classical case, our Taylor polynomial $\Tay_x^k f$ coincides of course
with the usual one, but we do not define it in the usual way in terms of higher
order differentials (since we then would have to divide
by $j !$, which is impossible in arbitrary characteristic). 
In a third step, we consider the algebraic scalar extension of the Taylor
polynomial from $\K$ to $\cN$:
if the degree $k$ is higher than the order of nilpotency of $\cN$, then
\begin{equation}
T^\bA_x f := (\Tay_x^k f)_{\cN} \, :V_\cN:=V\otimes\cN\to W_\cN
\end{equation}
satisfies a plainly functorial transformation rule, so that finally the Weil functor $T^\bA$
can be defined by $T^\bA f: U\times V_\cN \to W\times W_\cN,$ $(x,y)\mapsto\bigl( f(x) , T^\bA_x f (y) \bigr)$.
This strategy requires to develop some general tools on continuity and smoothness of polynomials,
which we relegate to Appendix C.

\msk From the point of view of analysis, the interplay between jets and Taylor polynomials
 is reflected by the interplay between {\em (generalized) difference quotient maps} and
 (various)  {\em remainder conditions} for  the remainder term in 
a ``limited expansion'' of a function $f$ around a point $x$.
The first aspect is functorial and well-behaved in arbitrary dimension, hence is suitable for
defining our general differential calculus;
it implies certain {\em radial limited expansions}, together with their
{\em multivariable versions}, which represent the second aspect, and which 
are the starting point for our definition of Taylor polynomials (Chapter 2).

\medskip
This work is organized in four main chapters and three appendices, as follows: 

\ssk
2. Taylor polynomials, and their relation with jets

3. Construction of Weil functors 

4. Weil functors as bundle functors on manifolds 

5. Canonical automorphisms, and graded Weil algebras 

Appendix A: Difference quotient maps and $\K^\times$-action

Appendix B: Differential calculi

Appendix C: Continuous polynomial maps, and scalar extensions

\msk
The results of the present work will allow a more conceptual and more
general approach to most of the topics treated in \cite{Be08}; this will be
investigated in subsequent work.

\bigskip
{\bf Notation.}
As already mentioned, $\K$ is  a commutative unital topological base ring, with dense unit group $\K^\times$, and
$\bA$ is a Weil $\K$-algebra. 
Concerning ``cubic'' and ``simplicial'' differential calculus, specific notation
is explained in Appendices A and B. 
In particular, superscripts of the form $[\cdot]$  refer to ``cubic'', and superscripts
of the form  $\la \cdot\ra$ refer to ``simplicial'' differential calculus. 
In general, the ``cubic'' notions are stronger than the ``simplicial'' ones
(``cubic implies simplicial'', see Theorem \ref{CubictosimplicialTheorem}).
As a rule, boldface variables are used for $k$-tuples or $n$-tuples, such as, e.g.,
$\vvv=(v_1,\dots,v_k)$, $\ooo = (0,\dots,0)$.

\section{Taylor polynomials, and their relation with jets}

\subsection{Limited expansions} 
Let $V,W$ be topological $\K$-modules and $f:U \to W$ a map defined on an open 
set $U \subset V$. Notation concerning {\em extended domains}, like $U^{[1]}$ and
$U^{\la k \ra}$, is explained in Appendix A. 

\begin{definition}
We say that $f$ admits  {\em radial limited expansions} if there exist
continuous maps $a_i:U \times V \to W$ and $R_k:U^{[1]} \to W$ such that, for $(x,v,t) \in U^{[1]}$,
$$
f(x+tv)=f(x)+\sum_{i=1}^kt^i a_i(x,v)+t^kR_k(x,v,t)
$$ 
and the
{\em remainder condition} $R_k(x,v,0)=0$ hold.
We say that $f$ admits  {\em multi-variable radial limited expansions} if
there exist continuous maps $b_i:U \times V^{i} \to W$ and $S_k: U^{\la k \ra} \to W$, such that,
for $\vvv = (v_1,\ldots,v_k) \in V^k$ with $x + \sum_{i=1}^k v_i \in U$,
$$
f(x+ t v_1 + t^2 v_2 + \ldots + t^k v_k) = 
f(x) + \sum_{i=1}^k t^i  b_i(x,v_1,\ldots,v_i) +t^k S_k(x,\vvv;0,\ldots,0,t)
$$
and  the {\em remainder condition $S_k(x,\vvv;0,\ldots,0,0)=0$} hold.
\end{definition}

Taking $0=v_2=v_3=\ldots=v_k$, the multi-variable radial condition  
implies the radial condition. For the next result,
cf.\ Appendix B for definition of the {\em class $C^{\la k \ra}$}.

\begin{theorem}[Existence and uniqueness of limited expansions]
\label{UniquenessTaylorExpansion}\label{LimitedCorollary}
Assume $f:U\to W$ is of class $\CC^{\la k\ra}$.
Then $f$ admits  limited expansions of both types from the preceding definition, 
and such expansions are unique, given by
$$
b_i(x,v_1,\ldots,v_i) = f^{\la i \ra}(x,v_1,\ldots,v_i; \ooo), 
$$
$$
a_i(x,h)= f^{\la i \ra}(x,h,\ooo;\ooo).
$$
\end{theorem}

\begin{proof}
Existence follows from Theorem
 \ref{LimitedTheorem}, by letting there
 $s_0 = s_1 = \ldots = s_{k-1} = 0$ and $s_k = t$. 
Uniqueness is proved as in  \cite{BGN04}, Lemma 5.2.
\end{proof}

Recall from Appendix B the definition of the class $C^{[k]}$, 
and the fact that
$C^{[k]}$ implies $C^{\la k \ra}$ (Theorem \ref{CubictosimplicialTheorem}).

\begin{corollary}\label{DiffSimpPolynomial}
Assume $f$ is of class $C^{[k]}$.
Then, for  $i=1,\ldots,k$, the {\em normalized differential}
$$
a_i(x,\cdot):V\to W, h \mapsto D^i_hf(x) :=a_i(x,h)
$$ 
is continuous and polynomial (in the sense of 
 Definition \ref{PolynomialDefinition}), homogeneous of degree $i$, hence smooth.
\end{corollary}

\begin{proof}
See  \cite{Be10}, Cor.\ 1.8. The last statement follows from Theorem \ref{PolContinousSmooth}.
\end{proof}

\subsection{Taylor polynomials}\label{Section:Taylor}
We can now define Taylor polynomials, which are the core of the construction of Weil functors.

\begin{definition}
Let $f:U \to W$ be of class $\CC^{[k]}$ and fix $x \in U$.
The polynomial
$$
\Tay_x^{k} f : V \to W , \quad h \mapsto \sum_{i=1}^k \D^i_h f(x) = \sum_{i=1}^k a_i(x,h) 
$$ 
is called the 
{\em $k$-th order Taylor polynomial of $f$ at the point $x$}.
\end{definition}

\begin{theorem}\label{TaylorPolyTheorem}
If $f:U \to W$ is of class $\CC^{[\ell]}$, then for all $k\leq \ell$,
$\Tay_x^{k} f$
is a smooth polynomial map of degree at most $k$,
without constant term.  
If, moreover, $f$ is itself a polynomial map of degree at most $k$, then
$\Tay_0^k f$ coincides with $f$, up to the constant term:
$$
f= f(0) + \Tay_0^k f \,  ,
$$
and its homogeneous part $f_i$ of $f$ of degree $i$ is equal to $a_i(x,\cdot)=
f^{\la i\ra}(x,\cdot,\ooo;\ooo)$.
\end{theorem}

\begin{proof} The fact that $\Tay_x^k f$ is smooth and  polynomial follows from Corollary
 \ref{DiffSimpPolynomial}.
 
 Next, assume that $f$ is a homogeneous
polynomial, say, of degree $j$ with $j \leq k$. Uniqueness of the radial Taylor expansion
(Theorem \ref{UniquenessTaylorExpansion}) implies that,
 for every homogeneous map of degree $j$ and of class $\CC^{[k]}$, we have
 $f(h)= f^{\la j \ra}(x,h,\ooo;\ooo)$ (see \cite{Be08}, Cor.\ 1.12 for the proof in case $j=2$,
 which generalizes without changes),
 whence the claim.
\end{proof}

The radial limited expansion can
now be written as
\begin{equation} \label{RadialEquation}
f(x+th)=f(x)+ \Tay_x^{ k}f(th) + t^k R_k(x,h,t) \, .
\end{equation}
If the integers are invertible in $\K$, then, by uniqueness of the expansion, it coincides 
with the usual Taylor expansion:
$\D^j_vf(x) = \frac{1}{j!} d^j f(x)(v,\ldots,v)$,
see \cite{BGN04}.

\subsection{Normalized differential and polynomiality}

\begin{definition}\label{NormalizedDifferentialDef} 
Let $f:U \to W$ be of class $\CC^{[k]}$. Recall from above the definition of
 the {\em normalized differential}
$\D^i_vf(x):=f^{\la i \ra}(x,v,\ooo;\ooo)$.
We define,
for all multi-indices $\Alpha\in\N^k$ such that $|\Alpha|:=\sum_i\Alpha_i\leq k$,
and for all $\vvv\in V^k$, $x \in U$, the {\em normalized polynomial differential}
(which is well-defined, by Corollary \ref{SimplicialDiffSmooth})
 $$
 \D^\Alpha_{\vvv}f(x):=\left( \D^{\alpha_k}_{v_k}\circ\ldots\circ \D^{\alpha_1}_{v_1} \right)f(x).
 $$
\end{definition}

The following is a generalization of Schwarz's Lemma (Th.\ \ref{SchwarzLemma} (iv)):

\begin{lemma}\label{NormalizedDifferentialSchwarzLemma}
Let $f:U\subset V\to W$ a a map of class $\CC^{[k]}$.
Then, for all multi-indices $\Alpha\in\N^k$ such that $|\Alpha|:=\sum_i\Alpha_i\leq k$,
and for all $\vvv\in V^k$,
 the map $\D^\Alpha_{\vvv}f(x)$ does not depend on the order
 in wich we compose the normalized differentials $\D^{\alpha_i}_{v_i}$.
\end{lemma}

\begin{proof}
$\D^\Alpha_{\vvv}f(x)$ is obtained, by restriction to $\sss = 0$, from the  map  
$$
\D^\Alpha_{\vvv,\sss}f(x):=\left( \D^{\alpha_k}_{v_k,\sss^{(k)}} \circ \ldots \circ \D^{\alpha_1}_{v_1,\sss^{(1)}} \right) f(x),
$$
where for all $1\leq i\leq k$, we let
$\D^{\alpha_i}_{v_i,\sss^{(i)}}:=f^{\la \alpha_i\ra}(x,v_i,\ooo;\sss^{(i)})$.
From the definition of the simplicial different quotient map, we get, for non-singular $\sss^{(1)}$:
$$
\D^{\alpha_1}_{v_1,\sss^{(1)}} f(x)=
\sum_{j_1=0}^{\alpha_1}
\frac{f\bigl(x +   (s^{(1)}_{j_1} - s^{(1)}_0) v_1\bigr)}
{\prod_{i=0,\ldots,\hat j_1,\ldots,\alpha_1 }(s^{(1)}_{j_1}-s^{(1)}_i)}.
$$
In a second step, for non-singular $\sss^{(2)}$, we get:
$$
\left( \D^{\alpha_2}_{v_2,\sss^{(2)}}  \circ \D^{\alpha_1}_{v_1,\sss^{(1)}} \right) f(x)=
\sum_{j_2=0}^{\alpha_2} \sum_{j_1=0}^{\alpha_1}
\frac{f\bigl(x +   (s^{(1)}_{j_1} - s^{(1)}_0) v_1 + (s^{(2)}_{j_2} -
s^{(2)}_0) v_2\bigr)}
{\prod_{i=0,\ldots,\hat j_1,\ldots,\alpha_1 }(s^{(1)}_{j_1}-s^{(1)}_i)
\prod_{i=0,\ldots,\hat j_2,\ldots,\alpha_2 }(s^{(2)}_{j_2}-s^{(2)}_i)},
$$
and so on: by induction, we finally get, for non singular $\sss^{(i)}$:
$$
\left( \D^{\alpha_k}_{v_k,\sss^{(k)}} \circ \ldots \circ \D^{\alpha_1}_{v_1,\sss^{(1)}} \right) f(x)=
\sum_{j_1=0}^{\alpha_1} \dots\sum_{j_k=0}^{\alpha_k}
\frac{f\bigl(x +  \sum_{\ell=1}^k (s^{(\ell)}_{j_\ell} - s^{(\ell)}_0)
v_\ell\bigr)}
{\prod_{\ell=1}^k\prod_{i=0,\ldots,\hat j_\ell,\ldots,\alpha_\ell
}(s^{(\ell)}_{j_\ell}-s^{(\ell)}_i)}.
$$
Obviously, the right-hand side term does not change if we apply the operators
$\D^{\alpha_i}_{v_i,\sss^{(i)}}$ in another order. 
Hence, by continuity and density of $\K^\times$ in $\K$, 
this remains true for singular values of $\sss^{(i)}$,
and in particular for $\sss^{(i)}=\ooo$.
\end{proof}

\begin{theorem}\label{NormalizedDifferentialPolynomial}\label{JetExplicitFormula}
Let $f:U\to W$ be a map of class $\CC^{[k]}$. 
Then:
\begin{enumerate}
\item
For all $x\in U$ and 
for all multi-indices $\Alpha\in\N^k$ such that $|\Alpha|\leq k$,
the map $V^k\to W,\vvv\mapsto \D^\Alpha_{\vvv}f(x)$ is 
polynomial multi-homogeneous of multidegree $\Alpha$.
\item
The map $\vvv \mapsto f^{\la j \ra}(x,\vvv; \ooo)$ is polynomial. More precisely,
 for all $1\leq j\leq k$,  
\begin{equation}\label{TaylorEqn}
f^{\la j \ra}(x,\vvv; \ooo)=\sum_{\Alpha\in\N^k, \sum_{i=1}^k i\alpha_i=j}\D^\Alpha_{\vvv} f(x) \ .
\end{equation} 
\end{enumerate}
\end{theorem}

\begin{proof}
(1) Note that
$
 \D^\Alpha_{\vvv} f(x)=\left(\D^{\alpha_k}_{v_k}\circ \ldots \circ \D^{\alpha_1}_{v_1}\right)f (x)
$ 
is polynomial homogeneous of degree $\alpha_k$ in $v_k$, by Corollary \ref{DiffSimpPolynomial}.
By the previous lemma, the value of $\D^\Alpha_{\vvv} f(x)$ is independent of
 the order in which  we compose the normalized differentials.
Therefore  $\D^\Alpha_{\vvv} f (x)$ 
is also polynomial homogeneous of degree $\alpha_i$ in $v_i$ for all $1\leq i\leq k$, 
i.e., $\vvv\mapsto \D^\Alpha_{\vvv} f(x)$ is a polynomial multi-homogeneous map
 of multidegree $\Alpha$.

\ssk
(2) On the one hand, we use the $k$-th order radial limited expansion, 
successively for each variable $v_i$, $1\leq i \leq k$. This is well-defined
 (see Corollary \ref{SimplicialDiffSmooth}).

\msk
\nin
$f\left(x+\sum_{i=1}^k t^iv_i\right) \ = \ f \left(x+\sum_{i=2}^{k} t^iv_i +t^1v_1\right)$
\begin{eqnarray*}
&=&\sum_{\alpha_1=0}^k t^{\alpha_1} \D^{\alpha_1}_{v_1}f 
\left(x+\sum_{i=2}^{k}t^iv_i\right)+t^kR^1_k(x,v_1,t)\\
&=&\sum_{\alpha_1=0}^k \sum_{\alpha_{2}=0}^{k-\alpha_1} t^{\alpha_1}t^{2\alpha_2}
(\D^{\alpha_2}_{v_2} \circ \D^{\alpha_1}_{v_1})f
  \left(x+\sum_{i=3}^k t^i\right)\\
 & & \qquad \qquad  +
t^kR^1_k(x,v_1,t)+
t^kR^{2}_k(x,v_1,v_2,t)\\
&=&\sum_{0\leq \alpha_1,\ldots,\alpha_k\leq k,\sum_i \alpha_i\leq k} 
t^{\alpha_1}\ldots t^{k\alpha_k}\left(\D^{\alpha_k}_{v_k}\circ \ldots \circ \D^{\alpha_1}_{v_1}\right)f (x)\\
&&\qquad   \qquad +\sum_{i=1}^k t^kR^i_k(x,v_1,\ldots,v_i,t)\\
&=&\sum_{\Alpha\in\N^k,|\Alpha|\leq k}t^{\sum_i i\alpha_i} \D^\Alpha_{\vvv} f (x)+t^kR_k(x,\vvv,t),\\
&=&f(x)+\sum_{j=1}^k\sum_{\Alpha\in\N^k,\sum_i i\alpha_i=j}t^{j} \D^\Alpha_{\vvv} f(x)+t^kR_k(x,\vvv,t),
\end{eqnarray*}
where $R_k(x,\vvv,t):=\sum_{i=1}^k R^i_k(x,v_1,\ldots,v_i,t) +
\sum_{j>k} \sum_{\Alpha \in \N^k, \sum_i i \alpha_i = j} t^{j-k} \D^\Alpha_{\vvv} f (x)$ 
satisfies the remainder condition 
$R_k(x,\vvv,0)=0$.

On the other hand, we use the $k$-th order multi-variable radial limited expansion:
$$
f(x+\sum_{i=1}^k t^iv_i ) \ = \
f(x) + \sum_{j=1}^k t^j f^{\la j \ra}(x,v_1,\ldots,v_j; \ooo) +t^k S_k(x,\vvv;0,\ldots,0,t)
$$
Relation (\ref{TaylorEqn})  now follows 
by uniqueness of this expansion (Theorem \ref{UniquenessTaylorExpansion}).
\end{proof}

\begin{remark}
If the integers are invertible in $\K$, then relation (\ref{TaylorEqn}) reads
\begin{equation}\label{TaylorEqn2}
f^{\la j \ra}(x,\vvv; \ooo)=\sum_{\Alpha\in\N^k, \sum_i i\alpha_i=j}
 \dfrac{d^{|\Alpha|}f(x,\vvv_\Alpha)}{\Alpha!} \, ,
\end{equation}
where $\Alpha!:=\alpha_1!\ldots\alpha_k!$ and 
$\vvv_{\Alpha}:=(\underbrace{v_1,\ldots,v_1}_{\alpha_1 \times  },\ldots,
\underbrace{v_k,\ldots, v_k}_{\alpha_k \times })$
($v_i$ appearing $\alpha_i$ times).
Indeed, $\D^{\alpha_1}_{v_1}f(x)=\frac{d^{\alpha_1}f(x)(v_1,\ldots,v_1)}{\alpha_1!}$;
iterating this, and using Theorem \ref{SchwarzLemma}, we get
\begin{eqnarray*}
\D^{\alpha_1,\alpha_2}_{v_1,v_2}f(x)=\left(\D^{\alpha_2}_{v_2}\circ \D^{\alpha_1}_{v_1}\right)f (x)
&=&\frac{d^{\alpha_2}\left( \frac{d^{\alpha_1}f(\cdot)(v_1,\ldots,v_1)}{\alpha_1!}\right)(x)(v_2,\ldots,v_2)}{\alpha_2!}\\
&=&\frac{d^{\alpha_1+\alpha_2}f(x)(v_1,\ldots,v_1,v_2,\ldots,v_2)}{\alpha_1!\alpha_2!},
\end{eqnarray*}
and so on: by induction, we have $f^\Alpha(x,\vvv)=\dfrac{d^{|\Alpha|}f(x,\vvv_\Alpha)}{\Alpha!}$,
whence (\ref{TaylorEqn2}).
Note that these formulae  are in keeping with the formulae given in
\cite{Be08}, Chapter 8; however, the methods used there are less 
well adapted to the case of arbitrary characteristic.
\end{remark}

\subsection{The simplicial $\K$-jet as a scalar extension.}
Recall from Appendix B the definition of the {\em (simplicial) $k$-jet of $f$}, $\SJ^k f$, and 
that $\SJ^{ \la k \ra}_{(\sss)}$ and in particular $\SJ^k$ are functors
(Theorem \ref{SimpDifferentialFunctorialityTheorem}). 
They commute with direct products,  and hence, applied to
the ring $(\K,+,m)$ with ring multiplication $m$ and addition $a$,
 they yield new rings, denoted by $\SJ_{(\sss)}^{\la k\ra}\K$, resp.\ $\SJ^k \K$.
These rings have been determined explicitly in \cite{Be10}:
\begin{equation}\label{JetRings}
\SJ^{\la k \ra}_{(\sss)} \K \cong \K[X]/\bigl(  X(X-s_1) \cdots (X-s_k) \bigr), \quad \quad
\SJ^k \K \cong \K[X]/(X^{k+1}).
\end{equation}
We denote the class of the polynomial $X$ in $\SJ^k\K$ by $\delta$, so that
$1,\delta,\ldots,\delta^k$ is a $\K$-basis of $\SJ^k \K$. 
The following facts can be proved in a conceptual way, without using the
explicit isomorphism: 
 
\begin{lemma}\label{GradedLemma}
The $\K$-algebra $\SJ^k \K$ is {\em $\N$-graded}, i.e., of the form
$$
\SJ^k \K = E_0 \oplus E_1 \oplus \ldots \oplus E_k \quad \mbox{ with } \qquad
E_j \cdot E_i \subset E_{i+j} .
$$
In particular,  $E_1 \oplus \ldots \oplus E_k$ is a nilpotent subalgebra.
\end{lemma}

\begin{proof}
This is a direct consequence of
 Theorem \ref{SimpDifferentialFunctorialityTheorem}: 
$\SJ^k m$ commutes with the canonical $\K^\times$-action; hence this action is
by ring automorphisms. Thus the eigen\-spaces 
$E_j = \{ x \in \SJ^k \K \mid \, \forall r \in \K^\times : r.x = r^j x \}$ define a grading of $\SJ^k \K$.
\end{proof}

The canonical projection
$$
\pi^k:\SJ^k \K \to \K, \quad [P(X)] \mapsto P(0)
$$
is a ring homomorphism having a
 section $\sigma^k : \K\to \SJ^k\K, \, t \mapsto t \cdot 1$ (classes of constant polynomials), 
 called the {\em canonical injection} or {\em canonical zero section}.
We denote by $\SJ^k_0 \K = \langle \delta, \ldots, \delta^k \rangle$ the  kernel of $\pi^k$ (the fiber of $\pi^k:\SJ^k\K\to \K$ over $0$).
Note that $\SJ^k  \K$ is again a topological ring having a dense unit group; hence
we can speak of maps that are smooth (or of class $\CC^{[k]}$) over this ring.

\begin{theorem} [Simplicial Scalar Extension Theorem]
\label{SimplicialExtensionTheorem}
If $f:U \to W$ is smooth over $\K$, then $f$ admits a unique extension to $\SJ^k\K$-smooth
map  $F:\SJ^k U \to \SJ^k W$:  there exists a unique map
$F:\SJ^k U \to \SJ^k W$ (namely, $F = \SJ^k f$) such that
\begin{enumerate}
\item
$F$ is smooth over the ring $\SJ^k \K$, and
\item
$F(x)=f(x)$ for all $x \in U$, that is,
$
F \circ \sigma_U = \sigma_W \circ f$,
where $\sigma_U:U \to \SJ^k U$ and $\sigma_W:W \to \SJ^k W$ are the canonical injections.
\end{enumerate}
More precisely, any $\SJ^k\K$-smooth
map  $F:\SJ^k U \to \SJ^k W$ is uniquely determined by its restriction to the base
$\sigma_U(U) \subset \SJ^k U$. 
\end{theorem}

\begin{proof}
Existence has been proved  in \cite{Be10}, Theorem 2.7.
Uniqueness is a consequence of Theorem \ref{JetExplicitFormula}: 
for $F$ as in the claim,
we will establish an ``explicit formula'' in
terms of its values on the base $\sigma_U(U)$. 
Let $z=(v_0,\ldots,v_k)=v_0+\delta v_1+\ldots + \delta^k v_k \in \SJ^k U$,
 with $x \in U$ and $v_i \in V$.
Since $F$ is smooth over the ring $\SJ^k \K$, we may take $t=\delta$ 
and use the radial expansion
of $F$ (cf.\ proof of Theorem  \ref{JetExplicitFormula}) at order $k+1$: we get
$$
F\left(x+\sum_{i=1}^k \delta^iv_i\right)=
F(x)+\sum_{\ooo\neq\Alpha\in\N^k}\delta^{\sum_i i\alpha_i} \D^\Alpha_{\vvv} F(x),
$$
where, in contrast to Theorem \ref{JetExplicitFormula}, 
no remainder term appears here, since $\delta^{k+1}=0$.

Now, it follows directly from the proof of Lemma \ref{NormalizedDifferentialSchwarzLemma}
that $\D^\Alpha_{\vvv} F(x)$ is determined by its values on the base (i.e., if
$ F(x)=0$ for all $x\in U$, then $\D^\Alpha_{\vvv} F(x) =0$):
for non-singular values of $\sss^{(i)}$, $1\leq i\leq k$, this is seen for the map
$\D^\Alpha_{\vvv,\sss}F(x)$
where we can take all $\sss^{(i)}\in \K^{\alpha_i}$, since the simplicial difference
quotients  are in terms of the values of $F$ at the points
$x+   \sum_{i=1}^k(s^{(i)}_{j_i} - s^{(i)}_0) v_{i} \in U$.
By density, uniqueness then also follows for singular values of $\sss$.
\end{proof}

\begin{corollary}\label{PolyCor}
Assume $P,Q:V \to W$ are $\K$-smooth polynomial maps. Then: 
\begin{enumerate}[label=\roman*\emph{)},leftmargin=*]
\item
 $\SJ^k P:\SJ^k V \to \SJ^k W$ is  a $\SJ^k\K$-smooth $\SJ^k\K$-polynomial map, and
 coincides with the algebraic scalar extension (cf.\ Appendix A, 
Definition \ref{PolynomialScalarExtension}) $P_{\SJ^k \K}$ of $P$ from $\K$ to
$\SJ^k \K$.
\item
The restriction
 $\SJ^k_0 P$
 of $ \SJ^k P$ to the fiber $\SJ^k_0V=V\otimes_\K \SJ^k_0\K$ over $0$  is again a  polynomial map, 
and it  coincides with the algebraic scalar extension of $P$ from $\K$ to
the (non-unital) ring $\SJ^k_0 \K$.
\item Assume $P(0)=Q(0)$. Then 
 $\SJ^k_0 P = \SJ^k_0 Q$ if, and only if, $P\equiv Q \mod(\deg > k)$ (i.e., $P$ and $Q$ coincide
 up to terms of degree $>k$).
 \end{enumerate}
\end{corollary}

\begin{proof} i) The extension $P_{\SJ^k \K}$ of $P$ from $\K$ to $\SJ^k\K$ is $\SJ^k\K$-smooth, 
$\SJ^k\K$-polynomial  and satisfies the extension condition: $$
P_{\SJ^k\K} \circ \sigma_U = \sigma_W \circ P,$$
 (see 
 Theorem \ref{PolynomialScalarExtensionSmooth}, with $\bA=\SJ^k\K$). 
By the preceding theorem, the $\SJ^k\K$-smooth map $\SJ^kP$ coincides with $P_{\SJ^k \K}$,
and thus  is  $\SJ^k\K$-polynomial.

Item ii) is proved by the same argument.
Finally, iii) follows from ii) since the algebraic scalar extension of a polynomial whithout constant term
$P$ from $\K$ to  $\SJ^k_0 \K$ vanishes if and only if $P$ contains no homogeneous terms
of degree $j=1,\ldots,k$.
\end{proof}

\subsection{Link between Taylor polynomials and simplicial jets}
It follows from Theorem \ref{JetExplicitFormula} that
 $\SJ^k f(x,\vvv)$ is polynomial in $\vvv$. We are going to show  
that this polynomial can be interpreted as a scalar extension of the
Taylor polynomial $\Tay_x^k f$:

\begin{theorem}[Scalar extension of the Taylor polynomial]
\label{TaylorTheorem}
Assume $f,g:
U \to W$ and $h:U'\to W'$ are of class $\CC^{[k]}$ such that $f(x)=g(x)$ and $h(U')\subset U$. Then:
\begin{enumerate}[label=\roman*\emph{)},leftmargin=*]
\item
$\Tay_x^{k} f= \Tay_x^{k} g$ $\iff$   $\SJ^k_x f = \SJ^k_x g.$
\item
The polynomial map $\SJ^k_x f$ is the 
scalar extension of the Taylor polynomial $\Tay_x^k f$ from $\K$
 to the nilpotent part $\SJ^k_0\K=\delta \K \oplus \ldots \oplus \delta^k \K$ of the ring $\SJ^k \K$:
$$
\SJ^k_x f = (\Tay_x^{k}f)_{\SJ^k_0 \K} \, .
$$
\item
We have the following ``chain rule for Taylor polynomials'':
$$
\Tay_x^k (g \circ h) = \bigl[ \Tay_{h(x)}^k g \circ \Tay_x^k h \bigr] \mod (\deg > k)
$$
(where $\mod (\deg > k)$ denotes the truncated composition of polynomials).
\end{enumerate}
\end{theorem}

\begin{proof} 
i), ``$\Leftarrow$'': 
assume $\SJ^k_x f = \SJ^k_x g$, then 
$$
\Tay_x^k f (v) = \sum_{i=1}^k f^{\la i\ra}(x,v,\ooo;\ooo) = \alpha \circ \SJ^k_x f (v,\ooo) 
= \alpha \circ \SJ^k_x f \circ \kappa (v)
$$
where 
$$
\alpha : \SJ^k_0W = W^k \to W, \quad (w_1,\ldots,w_k) \mapsto w_1 + \ldots + w_k
$$ 
and
$$
\kappa : V \to \SJ^k_0 V, \quad v \mapsto (v,\ooo) \, .
$$
It follows that $\SJ^k_x f = \SJ^k_x g$ implies $\Tay^k_x f = \Tay^k_x g$.

\ssk
i), ``$\Rightarrow$": assume that $\Tay_x^k f = \Tay_x^k g$.   Then for
 $\phi:= f-g$ we have $\phi(x)=0$ and  $\Tay_x^{k} \phi = 0$, i.e.,
$$
\forall j=1,\ldots,k, \,  \forall v \in V : \quad \phi^{\la j \ra}(x,v,\ooo;\ooo) = 0 .
$$
In order to prove that $\SJ^k_x \phi =0$, we have to show that $\phi^{\la j \ra}(v_0,\ldots,v_j;\ooo)=0$,
for all $\vvv \in \SJ^k U$.  This is done by  computing
 $\phi(x+tv_1+t^2 v_2 + \ldots + t^k v_k)$ in two different ways, using
first the radial limited expansion, and then the multi-variable radial limited
 expansion:
let $w:=v_1 + t v_2 + \ldots + t^{k-1} v_k$; since $\phi(x)=0$ and $\Tay^j_x \phi =0$ for $j=1,\ldots,k$, we get
\begin{eqnarray*}
\phi(x+tw) & =&
\sum_{j=0}^{k} t^j \phi^{\la j \ra}(x,w,\ooo;\ooo) + t^k (\phi^{\la k \ra }(x,w,\ooo;\ooo,t)- \phi^{\la k \ra }(x,w,\ooo;\ooo))
\cr
&=& t^k (\phi^{\la k \ra }(x,w,\ooo;\ooo,t)- \phi^{\la k \ra }(x,w,\ooo;\ooo)) \, .
\end{eqnarray*}
On the other hand, the multi-variable limited expansion gives, with $x=:v_0$,
$$
\phi(x+tv_1+\ldots + t^k v_k)=
\sum_{j=0}^{k} t^j \phi^{\la j \ra}(v_0,\ldots,v_j;\ooo) +
 t^k (\phi^{\la k \ra }(\vvv ; \ooo,t)-\phi^{\la k \ra }(\vvv ; \ooo)) .
$$
By uniqueness of the radial limited expansion, 
$\phi^{\la j \ra}(v_0,\ldots,v_j;\ooo) = 0$ for $j=1,\ldots,k$.

\ssk
(ii) 
Choose the origin in $V$ such that $x=0$.
Now let $f:U \to W$ be of class $\CC^{[k]}$ and let
$P:=\Tay_0^k f$.
Since $P$ coincides (up to the additive constant $P(0)=0$) with its own Taylor polynomial 
(Theorem \ref{TaylorPolyTheorem}),
it follows that $\Tay_0^k f = \Tay_0^k P$, whence, by (i),
$\SJ^k_0 f = \SJ^k_0 P$, and the latter is $\SJ^k_0\K$-polynomial,  and coincides with its algebraic
scalar extension from $\K$ to $\SJ^k_0 \K$ (Corollary \ref{PolyCor}). 
Note that the homogeneous parts of degree $> k$  vanish, hence $\SJ^k_x f$ is of degree at most $k$.

\ssk
(iii) Let
$R:=\Tay_0^{k}(g \circ h)$,
$P:=\Tay_{h(0)}^{k} g$,
$Q:=\Tay_0^{k}h$.
By (i), $\SJ^k_0 h = \SJ^k_0 Q$ and $\SJ^k_{h(0)}g=\SJ^k_{Q(0)}P$. Using this, and  functoriality of $\SJ^k$, we get
$$
\SJ^k_0R = \SJ^k_0(g \circ h) = \SJ_{h(0)}^k g \circ \SJ^k_0 h = \SJ^k_{Q(0)} P \circ \SJ^k_0 Q =
\SJ^k_0 (P \circ Q),
$$ 
whence, by Corollary \ref{PolyCor}, $R \equiv P \circ Q \mod (\deg > k)$.
\end{proof}

\section{Construction of Weil functors}

\subsection{Weil algebras}
The notion of  Weil algebra has been defined in the introduction (Definition
\ref{WeilalgebraDef}).  Weil algebras form a category:

\begin{definition}
A {\em morphism of Weil $\K$-algebras} is a continuous  $\K$-algebra homomorphism 
$\phi:\bA \to \bB$ preserving the nilpotent ideals:
$\phi(\cN_\bA) \subset \cN_\bB$.
The automorphism group of $\bA$ is denoted by $\Aut_\K(\bA)$.
\end{definition}

\begin{lemma} \label{WeilalgebraLemma}
The canonical projection $\pi^\bA:\bA \to \K$ of a Weil algebra is
continuous, and so is its section $\sigma^\bA:\K \to \bA$, $x \mapsto x\cdot 1$.
The unit group $\bA^\times$ is open and dense in $\bA$, and inversion
$\bA^\times \to \bA$ is continuous. 
%
\end{lemma}

\begin{proof}
Recall first that every element of the group $\Gl(n+1,\K)$ acts by homeo\-morphisms on
$\K^{n+1}$ (with product topology), hence the topology on $\bA = \K \oplus \cN$ is indeed independent
of the chosen $\K$-basis. The continuity of $\pi^{\bA}$ and of $\sigma^{\bA}$ is then clear.

An element $x + y \in \bA = \K \oplus \cN$ is invertible if, and only if, $x$ is invertible
in $\K$: indeed, its  inverse is given by 
$$
(x+y)\inv = x\inv \sum_{j=0}^r  (-1)^j (x\inv y)^j,
$$
where $r \in \N$ is such that $\cN^r = 0$.  Hence
$\bA^\times = \K^\times \times \cN$ is open and dense in $\bA$, and inversion is seen to be
continuous since inversion in $\K$ is continuous.
\end{proof}

\begin{example}
The {\em iterated tangent rings} and the {\em jet rings}, 
$$
T^k \K \cong \K[X_1,\ldots,X_k]/(X_1^2,\ldots, X_k^2) \cong
\K \oplus \bigoplus_{\alpha \in \{ 0, 1 \}^k, \alpha\not= 0} \eps^\alpha \K \, ,
$$
$$
\SJ^k \K \cong \K[X]/(X^{k+1}) \cong \K \oplus \bigoplus_{j=1}^k \delta^j \K \, ,
$$
(cf.\ (\ref{JetRings}))  are Weil $\K$-algebras. 
Note that the permutations of the elements $\eps^\alpha$, induced by the permutation
group $S_k$,  define natural Weil algebra automorphisms of $T^k \K$. 
The {\em canonical $\K^\times$-action} (Appendix A) on $T^k \K$ and on $\SJ^k \K$
is also by automorphisms. 
More generally, the following truncated polynomial algebras in several variables
are Weil algebras:
$$
\bW^r_k(\K):= \K[X_1,\ldots,X_k]/I_r
$$
where $I := I_0 := \langle X_1,\ldots,X_k \rangle$ is the ideal generated by all linear
forms and $I_r := I^{r+1}$ is the ideal of all polynomials of degree greater than $r$.
This is indeed a Weil algebra:
 as a $\K$-module, this quotient is the space of polynomials of degree at most
$r$ in $k$ variables, which is free. For $k=1$, we have
$\bW^r_1(\K)=\SJ^r \K$, in particular, $\bW^1_1 = T\K$. 
If $\bA$ is any Weil algebra, and $a_1,\ldots,a_n$ a 
$\K$-basis of $\cN$, then, if $\cN^{r+1} = 0$,
$$
\bW^r_{n}(\K) \to \bA = \K \oplus \cN, \quad  P \mapsto \bigl(P(\ooo),P(a_1,\ldots,a_n) \bigr)
$$
is well-defined and defines a surjective algebra homomorphism.  
Thus every Weil algebra is a certain quotient of an algebra $\bW^r_n(\K)$.
If $\K$ is a field, then such a representation with minimal $r$ and $n$
 is in a certain sense unique, with
$n = \dim (\cN / \cN^2)$ and $r$ the smallest integer with $\cN^{r+1}=0$
(see \cite{Kolar}, Sections 1.5 -- 1.7 for the real case; the arguments carry over to a general
field), but if $\K$ is not a field, this will no longer hold (for instance, $\K$ itself may then be
a Weil algebra over some other field or ring).  
It goes without saying that a ``classification'' of Weil algebras is completely out
of reach. 
\end{example}

\begin{lemma}\label{WhitneyLemma1}
Let $\bA = \K \oplus \cN$ and $\bB = \K \oplus \cM$ be two Weil algebras over $\K$.
\begin{enumerate}
\item
The tensor product $\bA \otimes \bB$ (where $\otimes = \otimes_\K$) 
is a Weil algebra over $\K$, with decomposition
$$
\bA \otimes \bB = \K \oplus (\cN \oplus \cM \oplus \cN \otimes \cM) \, .
$$
\item
The ``Whitney sum'' $\bA \oplus_\K \bB:=\bA \otimes \bB / \cN \otimes \cM$ 
is a Weil algebra over $\K$, with decomposition
$$
\bA \oplus_\K \bB \cong  \K \oplus (\cN \oplus \cM)
$$
\item
Both constructions are related by the following ``distributive law''
$$
\bA \otimes_\K \bigl( \bB \oplus \bB' \bigr) \cong
\bigl( \bA \otimes_\K \bB \bigr) \oplus_\bA  \bigl( \bA \otimes_\K \bB' \bigr)  \, .
$$
\end{enumerate}
\end{lemma}

\begin{proof}
The tensor product of two commutative algebras is again a commutative algebra,
and we have a chain of ideals
$\cN \otimes \cM \subset (\cN \oplus \cM \oplus \cN \otimes \cM) \subset \bA \otimes \bB$.
Since $\bA \otimes \bB$ is again free and finite-dimensional over $\K$, the product
topology is canonically defined on $\bA \otimes \bB$ and on the respective quotients.
The given decompositions are standard isomorphisms on the algebraic level, 
and by the preceding remarks  they are also homeomorphisms.
\end{proof}

\begin{example}
The tensor product
$T^k\K\otimes T^\ell \K$ is naturally isomorphic to $T^{k+\ell} \K$.
The direct sum
$T\K \oplus_\K \ldots \oplus_\K T\K$ ($n$ factors) is naturally isomorphic to
$\bW^1_n(\K)$ (the Weil algebra of ``$n$-velocities'').
The Weil algebra $T^k\K$ is a quotient of $\bW^1_k(\K)$. 
\end{example}

\subsection{Extended domains}
As a first step towards the definition of Weil functors, we have to define the
{\em extended domains} of open sets $U$ in a topological $\K$-module $V$.
The algebraic scalar extension $T^\bA V:=V_\bA:= V \otimes \bA$
decomposes as
$$
V_\bA = V \otimes (\K \oplus \cN) = V \oplus (V\otimes \cN) = V \oplus V_\cN,
$$
and, if $\cN$ is homeomorphic to $\K^n$ with respect to a $\K$-basis of $\cN$,
then $V_\cN$ is isomorphic, as topological $\K$-module, to $V^n$ with product
topology. The canonical projection and its section,
$$
\pi_V:=\pi^\bA_V: V_\bA =V \oplus V_\cN \to V,  \quad \quad
\sigma_V{:=\sigma^\bA_V}:V \to V \oplus V_\cN
$$
are continuous. More generally, for any non-empty subset $U \subset V$ we define
the {\em $\bA$-extended domain} to be
$$
T^\bA U := (\pi_V)\inv(U) = U \times V_\cN \subset V_\bA .
$$
For any $x \in U$, the set $T^\bA_xU:=T^\bA(\{ x \}) \cong V_\cN  \subset T^\bA U$ is called
the {\em fiber over $x$}. 

\ssk
Let $P:V \to W$ be a $\K$-polynomial map, of degree at most $k$. Let 
$P_\bA:V_\bA \to W_\bA$ be its scalar extension from $\K$ to $\bA$
and
$P_\cN:V_\cN \to W_\cN$ be its scalar extension from $\K$ to $\cN$. 
That is, if $P=\sum_{i=0}^k P_i$ with $P_i$ homogeneous of degree $i$, then
$$
P_\bA (v \otimes a) = \sum_{i=0}^k (P_i)_\bA (v \otimes a) =  \sum_{i=0}^k P_i (v) \otimes a^i \, 
$$
(cf.\ Appendix C).
Then $P_\bA$ {\em extends} $P$ in the sense that
$P_\bA(v \otimes 1)=P(v) \otimes 1$, {\it i.e.}, 
$$
P_\bA \circ \sigma_V = \sigma_W \circ P .
$$
Note that we have also
$P \circ \pi_V = \pi_W \circ P_\bA$. In the same way, we define
$P_\cN$; mind that there is no commutative diagram of sections, 
 as there is no natural section $\K\to\cN$.

\subsection{Construction of Weil functors}\label{WeilFunctor}
The following main result generalizes Theorem \ref{SimplicialExtensionTheorem}
  from $\SJ^k \K$ to the case of an
arbitrary Weil algebra $\bA$:

\begin{theorem}{\rm (Existence und uniqueness of Weil functors)}
\label{WeilFunctorTheoremExistence}
 Let $f:U \to W$ be of class $\CC^{[\infty]}$ over $\K$ and $\bA$ a
Weil $\K$-algebra. Then $f$ extends to an $\bA$-smooth map: there exists a map  $T^\bA f:T^\bA U\to T^\bA W$ such that:
\begin{enumerate}
\item
$T^\bA f$ is of class $\CC^{[\infty]}_\bA$,
\item 
$T^\bA f \circ \sigma_U = \sigma_W \circ f$, i.e.,
$T^\bA f (x,\ooo)=(f(x),\ooo)$ for all $x \in U$,
\item $\pi_W \circ T^\bA f = f \circ \pi_U$.
\end{enumerate}
The map  $T^\bA f$ is uniquely determined by properties (1) and (2): if
 $F:T^\bA U\to T^\bA W$ is  of class $\CC^{[\infty]}_\bA$ and such that
$F \circ \sigma_U = \sigma_W \circ f$, then
$F = T^\bA f$.
More generally, 
any $\bA$-smooth map $F:T^\bA U\to T^\bA W$ is entirely determined by
its values on the base $\sigma_U(U)$.
\end{theorem}

\begin{proof}
Let $f:U \to W$ be of class $\CC^{[k]}$. 
Assume $\bA = \K \oplus \cN$ is a Weil algebra with $\cN$ nilpotent of order $k+1$. 
For all $x \in U$, define
$$
T^\bA_x f := (\Tay_x^{k}f)_\cN : V_\cN \to W_\cN
$$
to be the scalar extension from $\K$ to $\cN$ of the Taylor polynomial $\Tay_x^k f$,
and let
$$
T^\bA f : T^\bA U \to T^\bA W, \quad
(x,z) \mapsto \bigl( f(x), T^\bA_x f (z) \bigr) \, .
$$
It satisfies property (3). Since $T_x^\bA f$ is  polynomial without constant term,
 property (2) of the theorem  is fulfilled. 
In order to prove property (1), we prove first that $T^\bA f$ is continuous:
first of all, according to Theorem \ref{PolynomialScalarExtensionSmooth} (Appendix C),
since $P:=\Tay_x^k f :V \to W$ is a continuous polynomial, its scalar extension
$P_\cN : V_\cN \to W_\cN, z\mapsto (\Tay_x^k f)_\cN(z)$ is continuous.
 By direct inspection of the proof of
Theorem \ref{PolynomialScalarExtensionSmooth}, one sees that the dependence on $x$
is also continuous, i.e.,
$(x,z) \mapsto  (\Tay_x^{k}f)_\cN(z)$ is again continuous, and hence
$(x,y) \mapsto T^\bA f (x,y)$ is continuous (cf.\ Remark \ref{ContinuousRemark}).

\ssk
Next we prove the functoriality rule $T^\bA (f \circ g)=
T^\bA f \circ T^\bA g$. For this, we use the ``chain rule for Taylor polynomials'' 
(Theorem \ref{TaylorTheorem}, Part (iii)),  together with nilpotency of $\cN$ and
the fact that, if $P$ is a polynomial containing only terms of degree $>k$, then
$P_\cN = 0$ by nilpotency. From this we get
\begin{eqnarray*}
T_x^\bA(g \circ f) & = &
(\Tay_x^k (g \circ f))_\cN \cr
&=& \bigl( \Tay_{f(x)}^k g \circ \Tay_x^k f \mod (\deg > k) \bigl)_\cN \cr
&=& \bigl( \Tay_{f(x)}^k g \circ \Tay_x^k f \bigl)_\cN \cr
&=& \bigl( \Tay_{f(x)}^k g \bigl)_\cN  \circ \bigr( \Tay_x^k f \bigl)_\cN \cr
&=& T_{f(x)}^\bA g \circ T_x^\bA f \, .
\end{eqnarray*}
Thus $T^\bA$ is a functor. It is obviously product preserving in the sense that
$T^\bA (f \times g)= T^\bA f \times T^\bA g$. 

\ssk
Now we can prove that $T^\bA f$ is smooth over $\bA$. In fact,
all arguments used for the proof of   \cite{Be08}, Theorem 6.2 (see also \cite{Be10}, Theorems 3.6,  3.7)
apply: $T^\bA (\K)=\K\otimes\bA$ is again a ring, and this ring is canonically isomorphic to $\bA$ itself;
 the conditions defining the class $\CC^{[1]}$ over $\K$ can be defined by a commutative diagram
invoking direct products, hence, applying a product preserving functor yields the same kind of
diagram over the ring $\bA= T^\bA \K$. One gets that
$$
(T^\bA f)^{]1[,\bA} = T^\bA (f^{]1[,\K}) .
$$
Since $f^{[1]}$ is smooth, $T^\bA (f^{[1],\K})$ is continuous by the preceding steps of the 
proof, hence $(T^\bA f)^{]1[,\bA}$ admits a continuous extension $(T^\bA f)^{[1],\bA}=T^\bA (f^{[1],\K})$, proving that
$T^\bA f$ is $\CC^{[1]}$ over $\bA$. By
induction, $f$  is then actually $\CC^{[\infty]}$ over $\bA$. 
 This proves the existence statement.
\ssk

Uniqueness is proved by adapting the method used in the proof 
of Theorem \ref{SimplicialExtensionTheorem}: 
fix a $\K$-basis $(1=a_0,a_1,\ldots,a_n)$  of $\bA$ and write an element of
$T^\bA U$ in the form $x+\sum_{i=1}^n a_iv_i$ with $x \in U$ and $v_i \in V$.
For $F = T^\bA f$, we develop in a similar way as in the proof of Theorem
\ref{NormalizedDifferentialPolynomial}, replacing the scalar $t^i$ by $a_i$
($i=1,\ldots,n$) and taking $k+1$-th order radial expansions: 
$$
F\bigl(x+\sum_{i=1}^n a_iv_i\bigr)=
F(x)+\sum_{\ooo\neq\Alpha\in\N^n}\aaa^\Alpha \D^\Alpha_{\vvv} F(x),
$$
where, as in the proof of Theorem  \ref{SimplicialExtensionTheorem}, no remainder term
appears, because of the nilpotency of $a_1,\ldots,a_n$.  
Since $x \in U$, we have by assumption $F(x)=f(x)$, and since all $v_i \in V$, as in the proof of  
Theorem \ref{SimplicialExtensionTheorem}, 
it follows 
that $\D^\Alpha_{\vvv} F(x) = \D^\Alpha_{\vvv} f(x)$, hence 
$T^\bA f$ is determined by its values on the base.
In the same way, we can develop any $T^\bA \K$-smooth map $F$, thus proving that
$F$ is determined by its values on the base $U$. 
\end{proof}

\section{Weil functors as bundle functors on manifolds}

\subsection{Manifolds}
Next we state the manifold-version of the preceding result. In order to fix
 terminology, let us recall  the definition of smooth manifolds:

\begin{definition}\label{ManifoldDefinition} 
Let $V$ be a topological $\K$-module, called the {\em model space} of the manifold.
A {\em (smooth) $\K$-manifold (with atlas, and modelled on $V$)} is a pair $(M,{\mathcal A})$, where $M$
is a topological space and $\mathcal A$ is a {\em $\K$-atlas of $M$}, i.e., an open covering
$(U_i)_{i\in I}$ of $M$, together with bijections $\phi_i:U_i \to V_i:=\phi_i(U_i)$ onto open subsets
$V_i \subset V$, such that the {\em chart changes}
$\phi_{ij}:=\phi_i\circ{\phi_j^{-1}}_{|\phi_j(V_{ji})}:V_{ji}\to V_{ij}$ are of class $\CC^{[\infty]}_\K$, where $V_{ij}:=\phi_i(U_i\cap U_j)$. Then
$\K$-smooth maps between manifolds are  defined in the usual way.
\end{definition}

For our purposes, it will be useful to assume always that a manifold is given {\em with atlas}
(maximal or not). 
The category of $\K$-manifolds will be denoted by $\Man_\K$. 
If $\bA$ is a Weil $\K$-algebra, then the category $\Man_\bA$ of smooth $\bA$-manifolds  is  well-defined.

\begin{theorem}{\rm (Weil functors on manifolds)}
\label{WeilFunctorTheoremManifold}
\begin{enumerate}
\item
There is a  unique functor $T^\bA:\Man_\K \to \Man_\bA$, which
coincides on open subsets $U$ of topological $\K$-modules with the assignment
$U \mapsto T^\bA U$, $f \mapsto T^\bA f$ described in Theorem \ref{WeilFunctorTheoremExistence}.
 Moreover, this functor is product preserving.
\item
The construction from (1) is functorial in $\bA$: 
if $\Phi:\bA \to \bB$ is a  morphism of Weil $\K$-algebras, then this defines canonically and
in a functorial way for all $\K$-manifolds $M$ a smooth map
$\Phi_M : T^\bA M \to T^\bB M$ such that, for all $\K$-smooth maps $f:M\to N$, we have
$$
T^\bB f \circ \Phi_M = \Phi_N \circ T^\bA f.
$$
\end{enumerate}
\end{theorem}

\begin{proof}
Recall that a manifold is equivalently given by the following data:
\begin{itemize}
\item a topological $\K$-module $V$ (the model space),
\item open sets $(V_{ij})_{i,j\in I}\subset V$, 
where $I$ is a discrete index set,
\item $\K$-smooth maps $(\phi_{ij})_{i,j\in I}$ (``chart changes'') satisfying the cocycle relations:
$$
\phi_{ii}=\id \ \mbox{and} \ \phi_{ij}\phi_{jk}=\phi_{ik} \ \mbox{ (where defined)}.
$$
\end{itemize}
We may then define the $\K$-manifold $M$ to be the set of equivalence classes $M:=S/\sim$, where 
$S:=\{(i,x)|x\in V_{ii}\}\subset I\times V$ and $(i,x)\sim(j,y)$ if and only if $\phi_{ij}(y)=x$, 
equipped with the quotient topology.
Conversely, we put 
$V_i:=V_{ii}$, $U_i:=\{[(i,x)],x\in V_i\}\subset M$ and $\phi_i:U_i\to V_i, [(i,x)]\mapsto x$ to recover the previous data.

Now we prove the existence statement in (1).
The functor $T^\bA$ associates to the topological $\K$-module $V$, to the 
open sets $V_{ij}\subset V$ and to the $\K$-smooth maps $\phi_{ij}$,
 the topological $\bA$-module $T^\bA V$, the open sets  $T^\bA V_{ij}\subset T^\bA V$ 
 and the $\bA$-smooth maps $T^\bA\phi_{ij}$. If $M$ is a $\K$-manifold  with model  $V$ 
 and atlas  $(V_{ij},\phi_{ij})$, then  $T^\bA V$ is a model and  $(T^\bA V_{ij},T^\bA\phi_{ij})$ 
 is an atlas of $\bA$-manifold. With those data, we have seen that we can construct 
 an $\bA$-manifold, denoted by $T^\bA M$. 

The proof of the uniqueness statement in (1) is obvious, since a manifold is entirely 
given by its model, charts domains and chart changes.

\ssk
(2) For open $U \subset V$ we define $\Phi_U : V \otimes \bA \supset T^\bA U \to V \otimes \bB$,
$v \otimes a \mapsto v \otimes \Phi(a)$. This is a $\K$-linear 
and continuous map, hence
 $\K$-smooth.
In particular, the collection of maps $\Phi_U:T^\bA U \to T^\bB U$ for chart domains $U$ defines
a well-defined smooth map $\Phi_M:T^\bA M \to T^\bB M$. 

Since $\Phi_V$ commutes with
scalar extension of polynomials ($P_\bB \circ \Phi_V = \Phi_W \circ P_\bA$),
it also commutes with extended maps, i.e.,
$T^\bB f \circ \Phi_M = \Phi_N \circ T^\bA f$.
\end{proof}

\begin{remark}
If $\Phi:\bA \to \bB$ is as in (2), then any $\bB$-module becomes an $\bA$-module by
$r. v := \Phi(r).v$. In this way, the target manifold $T^\bB M$ can also be seen as
a smooth manifold over $\bA$, in such a way that  $\Phi_M$ becomes $\bA$-smooth. 
This remark will be important for further developments in differential geometry  
(subsequent work).
\end{remark}

\subsection{Weil functors: bundle point of view}
Next we state the bundle version of the main theorem, and we give the formulation of 
certain operations
on Weil bundles in terms of their Weil algebras. 
The  precise definitions of notions related to 
bundles are explained after the statement of the results.

\begin{theorem} [Weil functors as bundle functors]
\label{WeilFunctorTheoremBundle}
Let $M \in \Man_\K$ modelled on $V$ and $\bA = \K \oplus \cN$ a Weil algebra such
that $\cN$ is nilpotent of order $k+1$. Then
\begin{enumerate}
\item
$T^\bA M$ is a $(\bA,\K)$-smooth polynomial bundle of degree $k$ with section  over $M$
and with  fiber modelled on $V_\cN = V \otimes_\K \cN$. 
More precisely, the chart changes are polynomial in fibers
of degree $k$ and without constant term.
In particular, if  $\cN^2 = 0$, then $T^\bA M$ is a  vector bundle over $M$.
\item
$T^\bA:\Man_\K \to \SBun^\bA_\K$ is the unique functor into
 bundles with section which
coincides on open subsets $U$ in topological $\K$-modules with the assignment
$U \mapsto T^\bA U$.
\item
If $\Phi:\bA \to \bB$ is a  morphism of Weil algebras, then $(\Phi_M,\id_M)$ is a
$\K$-smooth and intrinsically linear
bundle morphism between $T^\bA M$ and $T^\bB M$.
\end{enumerate}
\end{theorem}

\begin{theorem}[The ``$K$-theory of Weil bundles'']
\label{KTheoryTheorem}
Let $\bA=\K\oplus \cN$ and $\bB=\K\oplus \cM$ be $\K$-Weil algebras.

\begin{enumerate}
\item
The Weil functor defined by the Whitney sum $\bA \oplus_\K \bB$ of two Weil
algebras (cf.\ Lemma \ref{WhitneyLemma1})
 is naturally isomorphic to the Whitney sum of $T^\bA$ and $T^\bB$ over the base 
 manifold, i.e., for all $M \in \Man_\K$,
$$
T^{\bA \oplus_\K \bB} M \cong T^\bA M \times_M T^\bB M,
$$
where $\times_M$ denotes the bundle product over $M$. By transport of structure, this
 defines a structure of $\bA \oplus_\K \bB$-manifold on $T^\bA M \times_M T^\bB M$.
 \item
The Weil functor defined by the tensor product $\bA \otimes_\K\bB$ is isomorphic to the
composition $T^\bB \circ T^\bA$, and the typical fiber of $T^{\bA \otimes_\K\bB} M$ over
$M$ is $\K$-diffeomorphic to
$$
V_\cN \oplus V_\cM \oplus V_{\cN \otimes \cM} .
$$
\item
There is a natural bundle isomorphism $T^{\bA\otimes\bB} M\cong T^{\bB\otimes\bA}M$ called the {\em generalized flip}.
\item
There is a natural ``distributivity isomorphism'' of bundles over $T^\bA M$
$$
T^\bA (T^\bB M \times_M T^{\bB'} M) \cong T^\bA T^\bB M \times_{T^\bA M} T^\bA T^{\bB'} M .
$$
\end{enumerate}
\end{theorem}

We stress once again that the bundles $T^\bA M$ and $T^\bB M$ are in general {\em not}
vector bundles, so that there is no natural ``fiberwise notion of tensor product''.
Never\-theless, there exists some relation between the bundle  $T^{\bA \otimes \bB} M$
and what one might expect to be a ``fiberwise tensor product''; this question is closely 
related to the topic of {\em connections} and will be left to subsequent work. 
--
Before turning to the (easy) proofs of both theorems, let us give the relevant definitions:

\begin{definition}
An {\em $(\bA,\K)$-smooth fiber bundle} (with atlas) is given by:

\begin{enumerate}
\item a surjective $\K$-smooth map $\pi : E\rightarrow M$  from an $\bA$-smooth manifold $E$ 
(the \emph{total space}), onto a $\K$-smooth manifold $M$ (the \emph{base}),
\item a \emph{type}: an operation on the left $\mu :G\times F\rightarrow F$, $(g,y)\mapsto \rho(g)y$  of a group  $G$ 
(the \emph{structural group}) on a $\K$-smooth manifold $F$ (the \emph{typical fiber}),
\item a \emph{bundle atlas}, which induces the condition of \emph{local triviality}: there are
 \begin{itemize}
 \item a $\K$-manifold atlas $(U_i,\phi_i)_i$ of $M$, and
  \item  $\bA$-diffeomorphisms $\alpha_i : \pi^{-1}(U_i)\rightarrow V_i\times F$, called 
  \emph{bundle charts}, such that the following diagram commutes:
  \begin{equation*}
\xymatrix{ E\supset \pi^{-1}(U_i) \ar[r]^{\alpha_i} \ar[rd]^\pi & V_i\times F\ar[d]^{\phi_i^{-1}\circ \pr_{V_i}}  \\ & U_i }
\end{equation*}
 \end{itemize} 
 \item We require the bundle charts to be  \emph{compatible}, that is, for all {\em bundle chart changes}
 $$
 \alpha_{ij}:=\alpha_i\circ \alpha_j^{-1} : V_{ij}\times F\rightarrow V_{ij}\times F,
 $$ 
 there exist  maps $\gamma_{ij}: V_{ij} \rightarrow G$ 
 (\emph{transition functions}) satisfying $\alpha_{ij}(x,y)=(\phi_{ij}(x),\rho(\gamma_{ij}(x))y)$.
\end{enumerate}
A {\em bundle morphism}
 between two $(\bA,\K)$-smooth bundles $\pi:E\to M$ and $\pi':E'\to M'$ is, as usual, a pair of maps 
 $(\Phi,\phi)$, where $\Phi:E\to E'$ is an $\bA$-smooth map and 
 $\phi:M\to M'$ is a $\K$-smooth map such that  $\pi' \circ \Phi = \phi \circ \pi$.
Finally, a \emph{bundle with section} is a fiber bundle together with
a $\K$-smooth section $\sigma:M \to E$ of 
$\pi:E \to M$, and a {\em morphism of bundles with section} is a 
morphism of fiber bundles commuting with sections:
$\Phi \circ \sigma = \sigma' \circ \phi$.
\end{definition}

Note that if we fix $x\in V_{ij}$, then the maps $y\mapsto \rho(\gamma_{ij}(x))y$ are $\K$-diffeomorphisms, 
so that we can see  $G$ (in fact $\rho(G)$) as a subgroup of $\Diff_\K(F)$.
We do not require $\mu$ and $\gamma_{ij}$ to be smooth.
 If it is the case (in particular, if $G$ is a $\K$-Lie group), then the bundle is said to be
  \emph{strongly differentiable}.
Obviously,
$(\bA,\K)$-smooth bundles form a category,  denoted by $\Bun^\bA_\K$.
Bundles with section also form a category, denoted by $\SBun$.

\begin{definition}[Polynomial bundle] \label{PolyBundleDef} Let $V,W$ be topological $\K$-modules.
\begin{enumerate}
\item
A fiber bundle with atlas
 is called a {\em $\K$-polynomial bundle of degree $k$} if the typical fiber is a $\K$-module 
 and if the structural group $G$ acts polynomially of degree $k$ on the typical fiber $F$ 
(thus  $\rho(G)$ is then a subgroup of the polynomial group $\GPol_k(F)$,
see Definition \ref{PolynomialDefinition}), 
{\em i.e.},  if the bundle chart changes $\alpha_{ij}$ are $\K$-polynomial of degree $k$ in fibers.
In particular, an {\em affine bundle} is a polynomial bundle of degree $1$. If, moreover,
 the bundle chart changes are without constant term, then the bundle is a {\em vector bundle}.
\item
A map $f:E \to E'$ between fiber bundles with atlas is called
 {\em intrinsically $\K$-linear (resp. $\K$-polynomial)}
 if the typical fibers are $\K$-modules,
if it maps fibers to fibers and if, 
with respect to all charts from the given atlasses, the chart representation
of $f:E_x \to E'_{f(x)}$ is $\K$-linear (resp. $\K$-polynomial).
\end{enumerate}
\end{definition}

\begin{proof} {\em (of Theorem \ref{WeilFunctorTheoremBundle})}
(1) We have seen that $T^\bA M$ is an $\bA$-manifold.
Moreover, the Weil algebras morphisms $\pi^\bA:\bA\to \K$ and its section 
$\sigma^\bA :\K\to \bA,t\mapsto t\cdot 1$ induce $\K$-smooth morphisms  
$\pi^\bA_M:T^\bA M\to M$ and its section $\sigma_M^\bA:M\to T^\bA M$ 
by Theorem \ref{WeilFunctorTheoremManifold}.
Locally, over a chart domain $U$, $\pi^\bA$ is given by the linear map 
$T^\bA U = U \times V_\cN \to U$. 
The section $\sigma^\bA$ is locally given by $U \to U \times V_\cN$, $x \mapsto (x,0)$.

Let us show that this bundle is indeed locally trivial, with typical fiber $V_\cN$,
 and that the chart changes are polynomial in fibers.
The bundle atlas is given by the $\K$-manifold atlas $(U_i,\phi_i)$ of $M$ and by the $\bA$-diffeomorphisms
 $\alpha_i:(\pi^\bA)^{-1}(U_i)=U_i\times V_\cN\to V_i\times V_\cN$, 
$(x,y) \mapsto T^\bA f(x,y)=(f(x),(\Tay_x^k f)_\cN (y))$.
The maps $y\mapsto (\Tay_x^k f)_\cN (y)$ are $\K$-smooth polynomial, of degree at most $k$ and  without constant term,  
hence define an $(\bA,\K)$-smooth polynomial bundle over $M$. 
In particular, if
 $\cN$ is nilpotent of order $2$, then $T^\bA M\to M$ is a polynomial bundle of degree $1$ 
without constant term, that is, a vector bundle.

\ssk
(2) follows directly from the uniqueness statement in 
 \ref{WeilFunctorTheoremExistence}, and (3) from
  \ref{WeilFunctorTheoremManifold}.
\end{proof}
 
\begin{proof} {\em (of Theorem \ref{KTheoryTheorem})}.
(1)
Let $f:V \supset U \to W$ smooth over $\K$. 
The result follows from 
$$
T^\bA U\times_U T^\bB U=U\times V_\cN\times V_\cM=T^{\bA\oplus_\K \bB} U
$$  and applying part (2) of the preceding theorem.

\ssk
(2)
Let $f:V \supset U \to W$ smooth over $\K$. 
We have 
\begin{eqnarray*}
T^\bB (T^\bA U))=T^\bB(U\times V_\cN)&=&(U\times V_\cN)\times (V\times V_\cN)_\cM\\
&=&U\times V_\cN\times V_\cM\times V_{\cN\otimes\cM}\\
&=&U\times V_{\cN\oplus\cM\oplus\cN\otimes\cM}\\
&=&T^{\bA\otimes\bB}U
\end{eqnarray*} 
Applying twice Theorem \ref{WeilFunctorTheoremExistence} and noting that 
$\sigma_{T^\bA U}^{\bA}\circ\sigma_U^\bA=\sigma_U^{\bA\otimes\bA}$,
 it follows that
$$
T^\bB (T^\bA f): T^\bB (T^\bA U) \to (W \otimes_\K \bA) \otimes_\K \bB
$$ 
is an extension of $f$ that is smooth over the ring
$T^\bB (T^\bA \K) = T^\bB \bA = \bA \otimes_\bK \bB  $. 
Hence, by the uniqueness statement in Theorem \ref{WeilFunctorTheoremManifold}, 
$T^\bB (T^\bA f) = T^{\bA \otimes \bB} f$.

\ssk
(3) follows from the Weil algebra isomorphism $\bA\otimes\bB\cong \bB\otimes\bA$.

\ssk
(4)  follows from (1), (2), (3) and Lemma \ref{WhitneyLemma1} (3).
\end{proof}

\section{Canonical automorphisms, and graded Weil algebras}\label{sec:Endom}

\subsection{Canonical automorphisms}
Concerning the action of the ``Galois group'' $\Aut_\K(\bA)$, Theorem 
\ref{WeilFunctorTheoremBundle}
implies immediately that it acts canonically
by certain intrinsically $\K$-linear bundle automorphisms of $T^\bA M$, called
{\em canonical automorphisms}:
$$
\Aut_\K(\bA) \times T^\bA M \to T^\bA M, \quad (\Phi,u) \mapsto \Phi_M (u) \,  .
$$
This action commutes  with the natural action
of the diffeomorphism group $\Diff_\K(M)$:
$$
\Diff_\K(M) \times T^\bA M \to T^\bA M, \quad (f,u) \mapsto T^\bA f (u) \, .
$$
Here are some important examples of canonical automorphisms:

\begin{example}
For each $r \in \K^\times$,
there is a canonical automorphism $T\K \to T\K$,
$x + \eps y \mapsto x + \eps r y$. 
The corresponding canonical map $TM \to TM$ is multiplication by the scalar
$r$ in each tangent space.
This example generalizes in two directions:
\end{example}

\begin{example}
By induction, the preceding example yields an action of
$(\K^\times)^k$ by automorphisms on the iterated tangent manifold $T^k \K$.
The action of the diagonal subgroup $\K^\times$ then gives the canonical
$\K^\times$-action $\rho_{[\cdot]}$ described in Appendix B.
\end{example}

\begin{example}\label{HomogExample}
For each $r \in \K^\times$,
there is a canonical automorphism $\SJ^k \K \to \SJ^k \K$,
$P(X) \mapsto P(rX)$. 
The corresponding action of $\K^\times$ by automorphisms is the action $\rho_{\la\cdot\ra}$ described
in Appendix B.
It is remarkable that, in these cases, the canonical automorphisms can be traced back to
isomorphisms on the level of difference calculus
(Appendix A):

\begin{theorem}\label{JetMorphism} Let $r \in \K^\times$,
$\sss \in \K^k$ and $M$ a $\K$-manifold with atlas modelled on $V$.
There is a canonical bundle isomorphism
$$
\SJ^k_{(s_1,\ldots,s_k)} M \to \SJ^k_{(r\inv s_1,\ldots,r\inv s_k)} M
$$
given in all bundle charts by 
$\vvv = (v_0,\ldots,v_k) \mapsto r.\vvv = (v_0,r v_1,\ldots ,r^k v_k)$.
\end{theorem}

\begin{proof}
This is a restatement of the homogeneity property Theorem 
\ref{SimpDifferentialFunctorialityTheorem}  (ii) in
terms of the invariant language of  manifolds (cf.\ \cite{Be10} for notation).
\end{proof}

There is a similar result for the $\K^\times$-action on the  bundles $  T^k_{(\ttt)} M$.
For general automorphisms $\Phi$ of $\SJ^k \K$ or $T^k\K$, it seems to be difficult
or even impossible to 
realize them  as limit cases of a families of isomorphisms in difference calculus. 
\end{example}

\begin{example}
The map  $TT\K\to TT\K$,
$x+ \eps_1 y_1 +\eps_2 y_2 +\eps_1 \eps_2 y_{12} 
\mapsto x+ \eps_1 y_2 +\eps_2 y_1 +\eps_1 \eps_2 y_{12} $ is an automorphism,
called the {\em flip}.
The corresponding canonical diffeomorphism $TTM \to TTM$ is also called the {\em flip}
(see \cite{KMS}). By induction, we get an action of the symmetric group $S_k$
on $T^k M$ (see \cite{Be08}).
Recall (\cite{BGN04}) that the flip comes from {\em Schwarz's Lemma}, and that the proof of
Schwarz's Lemma in loc.\ cit.\ uses a symmetry of difference calculus in 
$\ttt = (t_1,t_2,t_{12})$ when $t_{12}=0$. It is not clear whether such a symmetry 
extends to difference calculus for all $\ttt$.

In a similar way, for any Weil algebra $\bA$, the map
 $\bA\otimes \bA\to \bA\otimes\bA, a\otimes a' \mapsto a' \otimes a$ is an automorphism, 
 called the {\em generalized flip}. The corresponding canonical diffeomorphism 
 $T^{\bA\otimes \bA}M \to T^{\bA\otimes \bA}M$ is also called the {\em generalized flip}.
\end{example}

\subsection{Graded Weil algebras and their automorphisms}

\begin{definition}
A Weil algebra $\bA = \K \oplus \cN$ is called {\em $\N$-graded (of length $k$)}  if it is of the
form $\bA = \bA_0 \oplus \ldots \oplus \bA_k$ with free submodules $\bA_i$ such
that $\bA_i \cdot \bA_j \subset \bA_{i+j}$ and
$\bA_0 = \K$.
\end{definition}

In \cite{KM04}, graded Weil algebras are called  {\em homogeneous Weil algebras}.
All examples of Weil algebras considered so far are graded -- in fact, 
it is not so easy to construct a Weil algebra that does not admit an $\N$-grading
(see \cite{KM04}) --
and in Lemma
\ref{GradedLemma} we have seen that such gradings arise naturally in differential calculus.
We are interested in graded Weil algebras because they admit a ``big'' group of 
automorphisms: if
we denote an element $a$ of $\bA=\bigoplus_{i=0}^k \bA_i$ by $(a_i)_{0\leq i\leq k}$, 
where $a_i\in \bA_i$ for all $i$, then obviously, for $r \in \K^\times$,
$$
\bA \to \bA, \quad (a_i)_{0\leq i\leq k} \mapsto (r^i a_i)_{0 \leq i \leq k}
$$
defines a ``one-parameter family''  of automorphisms, corresponding to the derivation 
$$
\bA \to \bA, \quad (a_i)_{0\leq i\leq k} \mapsto (i  a_i)_{0 \leq i \leq k} .
$$
This generalizes the canonical $\K^\times$-action on $T^k \K$ and on $\SJ^k \K$
from Example \ref{HomogExample}.
But there are other canonical maps:
recall that usual composition of formal power series without constant term,
$Q,P \in \K[[X]]_0$, 
is given by the following explicit formula,
for  $Q(X)=\sum_{n=1}^\infty b_n X^n$  and
$P(X)=\sum_{n=1}^\infty a_n X^n$: 
\begin{equation}\label{CompositionEquation}
Q \circ P (X) = \sum_{n=1}^\infty c_n X^n  \quad \mbox{with} \qquad
c_n = 
\sum_{j=1}^n b_j \sum_{\Alpha \in \N^j,  \vert \Alpha \vert = n} 
a_{\alpha_1} \cdots a_{\alpha_j} \, .
\end{equation}
Via the ``shift'' $S:\K[[X]] \to \K[[X]]_0$, $P(X) \mapsto X P(X)$, we transfer this law 
to $\K[[X]]$ and define a product
\begin{equation}\label{StarEquation}
Q \star P(X) := (\sum_{n=0}^\infty b_n X^n) \star  (\sum_{n=0}^\infty a_n X^n) :=
(S\inv (SQ \circ SP) )(X) = \sum_{n=0}^\infty u_n X^n
\end{equation}
with $u_n = \sum\limits_{j=0}^{n} b_{j} \sum\limits_{\alpha_0 + \ldots +
\alpha_j = n - j}a_{\alpha_0} \cdots a_{\alpha_j}$ (for $n=0$, this has to be interpreted as $u_0=b_0 a_0$).
Formally, $Q \star P = \frac{1}{X} ( (XQ) \circ (XP))$.
Now, the formula for $u_n$ is compatible with {\em graded} algebras, leading to the
following result:

\begin{theorem}\label{GradedTheorem}
Let $\bA$ be a graded Weil algebra of length $k$.
Then there is a well-defined ``product'' on $\bA$,
given by
$$
b \star a = (b_i) \star (a_i) := (u_i), \quad
u_i = \sum\limits_{j=0}^{i} b_{j} \sum\limits_{\alpha_0 + \ldots +
\alpha_j = i - j}
a_{\alpha_0} \cdots a_{\alpha_j} .
$$
This product is associative and right-distributive (i.e., $(b + b') \star a = b \star a + b' \star a$;
in other words, $(\bA,+,\star)$ is a {\em near-ring}). The right-multiplication operators
$$
R_a : \bA \to \bA, \qquad b \mapsto R_a(b):= b \star a,
$$
with $a \in \bA$, are algebra endomorphisms of the Weil algebra $(\bA,+,\cdot)$.
They are isomorphisms if, and only if, $a \in \bA^\times$, i.e., iff $a_0 \in \K^\times$.
\end{theorem}

\begin{proof}
The product is indeed well-defined: with notation from the theorem, we have indeed
$u_i \in \bA_i$. 
The map
$\Psi:\bA\to \bA[[X]]$, $(a_i)_i\mapsto \sum_{i=0}^k a_iX^i$
is a $\K$-linear map onto its image $\bA':=\bA_0\oplus\bA_1X\oplus\dots\oplus\bA_kX^k$.
It intertwines all algebraic structures considered so far: addition $+$, algebra product
$\cdot$ (here we use nilpotency of $\bA$) and $\star$ (as defined in the theorem, resp.\
by equation (\ref{StarEquation})). Therefore all claims now follow immediately from
known properties of the near-ring of formal power series $\bA[[X]]$.
\end{proof}


\begin{corollary}
With assumptions and notation as in the theorem, the group $(\bA^\times,\star)$ acts by canonical
and intrinsically linear automorphisms on each Weil bundle $T^\bA M$, and the monoid
$(\bA,\star)$ acts by intrinsically linear bundle endomorphisms. 
\end{corollary}

\begin{example}
If  $a=a_0=r\in\K^\times$, the operators $R_a$ give us again the canonical 
$\K^\times$-action.
\end{example}

\begin{example}
Let $\bA = T^k \K = \K \oplus \bigoplus_{\alpha} \eps^\alpha \K$ (cf.\ \cite{Be08},
Chapter 7), and let $a$ such that $a_i = 0$ for $i \not=1$ and
$a_i = \eps_j$ for a fixed $j \in \{ 1,\ldots , k \}$.
Then $R_a$ is the {\em shift operator} denoted by $S_{0j}$ in \cite{Be08}, Chapter 20.
\end{example}

\begin{example}
Similarly, for $\bA = \SJ^k \K =\K[X]/(X^{k+1})$, with 
$a=a_1=\delta$, we get a ``shift'' $[P(X)]\mapsto [P(X^2)]$.
\end{example}

\appendix

\section{Difference quotient maps and $\K^\times$-action}

In this appendix we recall some basic definitions concerning
difference calculus from \cite{BGN04} and \cite{Be10},
and we emphasize the fact that the group $\K^\times$ acts, in a natural way, on
all objects. 
In this appendix, $\K$ may be any commutative unital ring and $V$ any $\K$-module
(no topology will be used).

\subsection{Domains and $\K^\times$-action}
 Let $U\subset V$ be a non-empty set, called  ``domain''. 
We define two kinds of  ``extended domains'', the cubic one, denoted by $U^{[k]}$ and the simplicial one, denoted by $U^{\la k \ra }$ 
for $k \in \N$, which will later be used as domains of definition of generalized kinds of tangent
maps, for a given map defined on $U$. By convention, $U^{[0]}:=U=:U^{\la 0\ra}$.

\begin{definition}[``cubic domains'']
The {\em first extended domain} of $U$ is defined as
$$
U^{[1]}:=\{ (x,v,t) \in V \times V \times \K | \,
x \in U, x+tv \in U \} \, .
$$
We say that $U$ is the {\em  base} of $U^{[1]}$, and the maps
$$
\pi^{[1]} : \, U^{[1]} \to U, \, \, (x,v,t) \mapsto x,  \quad \mbox{ and its section } \quad 
\sigma^{[1]} : \, U \to U^{[1]}, \, \, x \mapsto (x,0,0)
$$
are called the {\em canonical projection}, resp.\ {\em injection}.
We call 
$$
U^{]1[}:=\{ (x,v,t) \in U^{[1]}\vert \, t \in \K^\times \}
$$
 the set of {\em non-singular
elements in the extended domain}.
Letting $t=0$, we define the {\em most singular set} or {\em tangent domain}
$$
TU:=\{(x,v,0) \in U^{[1]}\} \cong  U \times V.
$$
By induction, we define the {\em higher order  extended cubic domains} (resp., the set of their {\em
non-singular elements}) for $k\in\N^{\star}$ by
$$
U^{[k+1]}:=(U^{[k]})^{[1]}, \quad U^{]k+1[}:=(U^{]k[})^{]1[}\,  ,
$$
and we let $T^{k+1} U := T(T^k U)$. There are {\em canonical projections}  $\pi^{[k]}_{[j]}:U^{[k]}\to U^{[j]}$, 
and their sections $\sigma^{[k]}_{[j]}:U^{[j]}\to U^{[k]}$ called {\em canonical injections}, for all $j\leq k$. 
Note that 
$$
U^{[2]} \subset (V^2 \times \K) \times (V^2 \times \K) \times \K \cong V^4 \times \K^3,
$$
and similarly we will consider
$U^{[k]}$ as a subset of $V^{2^k}\times \K^{2^{k} - 1}$
and identify $T^k U$ with $U \times V^{2^k - 1}$. 
Elements of $V$ will be called ``space variables'',  and elements of $\K$ will be called ``time variables''.
We separate  space variables and  time variables. 
Correspondingly, we denote elements of $U^{[k]}$ by 
$(\vvv,\ttt)=\bigl( (v_\alpha)_{\alpha\subset \{1,\ldots,k\}},(t_\alpha)_{\emptyset \not= \alpha \subset \{1,\ldots,k\}} \bigr)$. 
With this notation, we have in particular, $v_\emptyset \in U$,
and $T^k U = \{  \bigl( \vvv,\ooo \bigr) \in U^{[k]} \}$.
\end{definition}

An explicit description of the extended domains for $k>1$ by 
conditions in terms of sets is fairly complicated, as the number of variables growths exponentially.
In order to get a rough understanding of their structure, 
it is useful to note a sort of ``homogeneity property''.

\begin{definition}
The {\em zero order action of $\K^\times$ on $U$} is the trivial action
$\K^\times \times U \to U$, $(r,x) \mapsto x$. 
The {\em canonical $\K^\times$-action on the first extended domain} is given by
$$
\rho_{[ 1]}: \K^\times \times U^{[1]} \to U^{[1]}, \quad (r,(x,v,t)) \mapsto \rho_{[ 1]}(r) .(x,v,t):=(x,rv,r\inv t) .
$$
This is  indeed well-defined: $x+tv \in U$ if and only if
$x + r\inv t rv \in U$. Moreover, the sets $U^{]1[}$ and $TU$ are stable under this action.
Next, the {\em canonical action of $\K^\times$ on $U^{[2]}$} is defined as follows:
write $(x,u,t) = ((v_\emptyset,v_1,t_1),(v_2,v_{12},t_{12}),t_2) \in U^{[2]}$ with 
$x \in U^{[1]}$, $u \in V^{[1]}$ and $t \in \K$ such
that $x+tu \in U^{[1]}$. For $r \in \K^\times$, let
\begin{eqnarray*}
\rho_{[ 2]}(r).\bigl((v_\emptyset,v_1,t_1),(v_2,v_{12},t_{12}),t_2\bigr) &:=&
\bigl( \rho_{[ 1]}(r). x, r \rho_{[ 1]}(r). u, r\inv t \bigr)
\cr
&=& 
\bigl( (v_\emptyset ,rv_1,r\inv t_1),(rv_2,r^2v_{12}, t_{12}), r\inv t_2 \bigr) 
\end{eqnarray*}
By induction, 
we define the {\em canonical action} $\rho_{[k+1]}:\K^\times \times U^{[k+1]} \to U^{[k+1]}$ via
$$
\rho_{[ k+1]}(r).\bigl( x,u,t \bigr) :=
\bigl( \rho_{[ k]}(r). x, r \rho_{[k]}(r). u, r\inv t \bigr),
$$
which can also be written as
$$
\rho_{[ k+1]}(r).\bigl( (v_\alpha)_{\alpha},(t_\alpha)_{\alpha \not= \emptyset} \bigr) :=
\bigl( (r^{\vert \alpha \vert} v_\alpha )_{\alpha} , (r^{\vert \alpha \vert - 2 } 
t_\alpha)_{\alpha \not= \emptyset} 
\bigr) ,
$$
where $|\alpha |$ is the cardinality of the set $\alpha\subset\{1,\ldots,k+1\}$.
The sets $U^{]k+1[}$ and $T^{k+1} U$ are stable under this action.
\end{definition}

The canonical projections and injections are equivariant with respect to this action.
Note that the operator $\rho_{[ k]}(r)$ is $\K$-linear. 
An element of $U^{[k]}$ will be called {\em homogeneous of degree $\ell$} if it is an eigenvector for
all $\rho_{[ k]}(r)$, where $r \in \K^\times$, with eigenvalue $r^\ell$. For instance, elements $x$ in
the base $U$ are homogeneous of degree zero. 
Now we define another kind of extended domain and explain its relation to the ones considered
above:

\begin{definition}[``simplicial domains'']
 Let $U \subset V$ be non-empty, and let (in all the following) $s_0 := 0$.
For $k \in \N^{ \star}$, we define the {\em
extended simplicial domains} 
\begin{eqnarray*}
U^{\langle 1\rangle} &:=&U^{[1]},
\cr
U^{\langle 2 \rangle} 
&:=& 
\Bigsetof{(v_0,v_1,v_2;s_1,s_2) \in V^3 \times \K^2 }
{\begin{array}{c}
 v_0 \in U, v_0 + (s_1 - s_0) v_1 \in U, 
 \\
 v_0 + (s_2-s_0) v_1 + (s_2 - s_1)(s_2-s_0) v_2 \in U
\end{array}}\,
\cr
U^{\langle k \rangle}
&:=&
\big\{ (\vvv ; \sss) \in V^{k+1} \times \K^{k} | \, 
v_0 \in U, \quad \forall i=1,\ldots,k : \,
v_0+\sum_{j=1}^i \prod_{m=0}^{j-1} (s_i - s_m) v_j \in U \big\}.
\end{eqnarray*}
Its set of {\em  non-singular elements} is
$$
U^{\ra k \la }:= \{ (\vvv;\sss) \in U^{\langle k \rangle} \mid \forall i \not= m: s_i - s_m \in \K^\times \} .
$$
Its set of  {\em most singular elements} is
$$
\SJ^k U := \{ (\vvv;\ooo) \in U^{\langle k \rangle} \} \cong U \times V^k .
$$
For $j<k$, there are obvious {\em canonical projections} and {\em injections}
$$
\pi^{\la k\ra}_{\la j\ra} : \,U^{\la k \ra } \to U^{\la j\ra}, \quad 
\mbox{ and its section } \quad \sigma^{\la k\ra}_{\la j\ra} : \, U^{\la j \ra} \to U^{\la k \ra} \, .
$$
For the subset of most singular elements, there are also the canonical projection $\pi^k:J^kU\to U$ and its section 
$\sigma^k:U\to \SJ^kU$.
\end{definition}

Compared to the cubic case, this
definition has two advantages: it is ``explicit'', and the number
of variables grows linearly instead of exponentially; its drawback is that it is not inductive. This will
be overcome by imbedding the simplicial domains into the cubic ones (Lemma 
\ref{DomainImbLemma} below). 
Note that in \cite{Be10}, $s_0$ has been considered as a variable. Since all ``simplicial
formulas'' invoke only differences $s_i - s_j$, this variable may be frozen to the value $s_0=0$, as
done here. There is an obvious $\K^\times$-action:

\begin{definition}
We define the {\em canonical action} $\rho_{\la k\ra}$ of $\K^\times$ on $U^{\la k \ra}$ by
$$
\rho_{\la k\ra}(r). \bigl( v_0,v_1,\ldots,v_k; s_1,\ldots,s_k \bigr) := \bigl( v_0,r v_1,r^2 v_2,\ldots ,r^k v_k ;
r\inv s_1,\ldots,r\inv s_k \bigr) \, .
$$
It is immediately seen that $\rho_{\la k\ra}(r)(\vvv;\sss) \in U^{\la k \ra}$ if and only if $(\vvv;\sss)\in U^{\la k \ra}$, 
 and that $U^{\ra k \la}$ and $\SJ^k U$ are stable
under this action.
\end{definition}

Like in the cubic case, projections and injections are $\K^\times$-equivariant, and
we may speak of {\em homogeneous elements (of degree $\ell$)}.
 Again, elements from the base $U$ are
homogeneous of degree zero. An important difference with the cubic case
is that here, in contrast to the cubic case,
scalars are always homogeneous of the same degree $-1$. 
Next, we define an equivariant imbedding into
the cubic domains:

\begin{lemma}\label{DomainImbLemma}
The map $g_k: U^{\la k\ra } \to  U^{[k]}$ defined by
$g_k(\vvv;\sss) := (\uuu, \ttt)$
with
$$
u_\alpha = \Biggl\{ 
\begin{matrix} v_0 & \mbox{ if } & \alpha = \emptyset \cr
v_i & \mbox{ if } & \alpha = \{ 1,\ldots, i \} \cr
0 & \mbox{else} & \end{matrix}
\quad \quad \quad \quad
t_\alpha = \Biggl\{ 
\begin{matrix} 1 & \mbox{ if } & \alpha = \{ i,i+1 \} \cr
s_i - s_{i-1} & \mbox{ if } & \alpha = \{ i \} \cr
0 & \mbox{else} & \end{matrix}
$$
is a well-defined, $\K$-affine and $\K^\times$-equivariant imbedding of $U^{\la k \ra }$ into $U^{[k]}$.
Moreover, $g_k(U^{\ra k \la }) \subset U^{]k[}$.
\end{lemma}

\begin{proof}
The fact that $g_k(U^{\la k \ra }) \subset U^{[k]}$ is directly checked (and follows also from
the recursion procedure used in \cite{Be10}, Lemma 1.5). 
In order to check the $\K^\times$-equivariance  $g_{k}\circ\rho_{\la k\ra}(r)= \rho_{[k]}(r)\circ g_k$,
note that on the level of space variables $\vvv$, homogeneous elements $v_i$ of degree $i$
are sent again to homogeneous elements of degree $i$ (since $\vert \{ 1,\ldots, i \} \vert = i$).
On the level of time variables $\sss$,  homogeneous elements $s_i$ of degree $-1$
are sent again to homogeneous elements of degree $-1$ (since $\vert \{  i \} \vert = 1$)
 and, for $\vert \alpha \vert = 2$, the $\K^\times$-action on 
homogeneous scalars is trivial. Altogether, this implies the equivariance.
The map $g_k$ is clearly injective: the inverse (of its corestriction to $g_k(U^{\la k\ra}$) is given by: 
${g_k}_{|g_k(U^{\la k\ra})}^{-1}(\uuu,\ttt) := (\vvv; \sss)$
with
$$
 \Biggl\{ 
\begin{matrix} v_0&=&u_\emptyset &  & \cr
v_i&=&u_{\{1,\ldots,i\}} & \mbox{for all } i>1 &  \cr
 \end{matrix}
\quad \quad \mbox{ and }\quad \quad
 \Biggl\{ 
\begin{matrix} s_0&=&0 & & \cr
s_i&=&\sum_{j=0}^i t_{\{j\}}  
\end{matrix}
$$
Note that this inverse is also $\K$-affine.
\end{proof}

The proof shows that the scalar components
 $t_\alpha$ with $\vert \alpha \vert = 2$ play a special r\^ole since
they are the only ones that are homogeneous of degree zero; one may say that they are
a sort of ``pivots''. Related to this, note that $g_k$ does {\em not} map $\SJ^k U$ to
$T^k U$.  
The imbedding $g_k$ has been used in
\cite{Be10}, Theorem 1.6. For $k=1$, $g_1$ is the identity, and for  $k=2,3$ , we have explicitly
$$
g_2(v_0,v_1,v_2;s_1,s_2) =
\bigl( (v_0,v_1,s_1),(0,v_2,1),s_2-s_1 \bigr) \, ,
$$
$$
g_3(v_0,v_1,v_2,v_3;s_1,s_2,s_3) =
\Bigl( \bigl( (v_0,v_1,s_1),(0,v_2,1),s_2-s_1 \bigr),
\bigl( (0,0,0),(0,v_3,0),1 \bigr),
s_3 - s_2 \Bigr) \, .
$$

\subsection{Difference calculus}
Let $V,W$ be $\K$-modules, 
$U \subset V$ a non-empty set and $f:U \to W$ a map. 
We first define ``cubic'' difference quotients and then ``simplicial'' ones, also called
{\em generalized divided differences}. The map $g_k$ will imbed them into the cubic calculus. 

\begin{definition}[``cubic difference quotients'']
The {\em first order difference quotient of $f$} is the map
$$
f^{]1[}:U^{]1[} \to W, \quad (x,v,t) \mapsto \frac{f(x+tv)-f(x)}{t} ,
$$
and the {\em extended tangent map} is the map
$$
  T^{]1[}f :U^{]1[} \to W^{]1[}, \quad (x,v,t) \mapsto (f(x),f^{]1[}(x,v,t),t) .
$$
The {\em higher order cubic difference quotients} and {\em higher order extended tangent 
maps} are defined by induction
\begin{eqnarray*}
f^{]k+1[} &:= &\left(f^{]k[}\right)^{]1[}:U^{]k+1[} \to W
\cr
  T^{]k+1[}f  &:=&  \left(  T^{]k[}\right)^{]1[} :U^{]k+1[} \to W^{]k+1[} .
\end{eqnarray*}
\end{definition}
In \cite{Be10}, explicit formulae for $f^{]2[}$ and $  T^{]2[} f$ are given;
they are quite complicated and not very useful. 
Here are two main properties of this construction:

\begin{theorem} \label{DifferenceFunctorTheorem}
Let $f:U \to W$ and $g:U' \to W'$ with $f(U)\subset U'$.
Then 
\begin{enumerate}[label=\roman*\emph{)},leftmargin=*]
\item
Functoriality: $  T^{]k[}(g\circ f) =   T^{]k[}g \circ  T^{]k[}f$ and $T^{]k[} \id_U =   \id_{T^{]k[}U}.$
\item
Homogeneity: for all $r \in \K^\times$, $  T^{]k[}f \circ \rho_{[k]}(r)=\rho_{[k]}(r) \circ   T^{]k[}f$.
\end{enumerate}
\end{theorem}

\begin{proof}
(i) For $k=1$, this is an easy computation, and for $k>1$, it follows immediately by induction
(see \cite{BGN04}).
(ii) For $k=1$:
$$
  T^{]1[}f (x,rv,r\inv t)= \left( f(x), \frac{f(x+r\inv t rv) - f(x)}{r\inv t},r\inv t \right) 
= \bigl( f(x), r f^{]1[}(x,v,t),r\inv t \bigr)
$$
and for $k>1$, the result follows by induction. 
\end{proof}

Now we come to simplicial difference calculus and to its imbedding into cubic calculus.
In the following, recall that, by definition, $s_0=0$.

\begin{definition}[``simplicial difference quotients'']
For a map $f:U \to W$ 
 we define  {\em (generalized) divided differences}  $f^{\rangle k \langle}:U^{\rangle k \langle}\to W$ by
\begin{eqnarray*}
f^{\rangle 1 \langle}(v_0,v_1;s_1) & := & f^{]1[}(v_0,v_1,s_1) = \frac{f(v_0)}{s_0-s_1} +
\frac{f(v_0+(s_1 - s_0)v_1)}{s_1-s_0}
\cr
f^{\rangle 2 \langle}(v_0,v_1,v_2;s_1,s_2) & := &
\frac{f(v_0)}{(s_0-s_1)(s_0-s_2)} +
\frac{f(v_0+(s_1 - s_0)v_1)}{(s_1-s_0)(s_1-s_2)} + \cr
& & \quad \quad \quad \quad \quad \quad
\frac{f(v_0+(s_2 - s_0)v_1 + 
(s_2 - s_1)(s_2-s_0)v_2)}{(s_2-s_0)(s_2-s_1)} \, ,
\cr
f^{\rangle k \langle}(\vvv;\sss )& := &
\frac{f(v_0)}
{\prod_{j=1,\ldots,k}(s_0-s_j)}
+ 
\sum_{i=1}^k
\frac{f\bigl(v_0 + \sum_{j=1}^i \prod_{m=0}^{j-1} (s_i - s_m) v_j\bigr)}
{\prod_{j=0,\ldots,\hat i,\ldots,k }(s_i-s_j)} .
\end{eqnarray*}
Define the {\em extended $k$-jet} by 
$\SJ^{\ra k \la } f: U^{\ra k \la } \to   W^{\ra k \la }$, 
$$
(\vvv;\sss) \mapsto \SJ^{\ra k \la } f(\vvv;\sss):=
\bigl( 
f(v_0),
f^{\rangle 1 \langle}(v_0,v_1;s_1),\ldots , 
f^{\rangle k \langle}(v_0,\ldots,v_k;s_1,\ldots,s_k) ; \sss) .
$$
\end{definition}

\begin{theorem}\label{DifferenceImbThm}
The map $g_k:U^{\la k \ra } \to U^{[k]}$ defines an imbedding of simplicial into cubic
difference calculus in the sense that, for all $f:U \to W$, we have
$$
  T^{]k[} f \circ g_k = (-1)^k  g_k \circ   \SJ^{\ra k \la } f : \quad \quad \quad 
\begin{matrix}
U^{\ra k \la } & \stackrel{\SJ^{\ra k \la } f} {\longrightarrow} & W^{\ra k \la } \cr
g_k \downarrow \phantom{g_k} & & (-1)^k g_k \downarrow \phantom{(-1)^k g_k} \cr
U^{]k[} &  \stackrel{  T^{]k[} f} {\longrightarrow}  & W^{]k[}
\end{matrix}
$$
\end{theorem}

\begin{proof}
See \cite{Be10}, Lemma 1.5 and Theorem 1.6.
\end{proof}

\begin{theorem} \label{SimpDifferenceFunctorialityTheorem}
Let $f:U \to W$ and $g:U' \to W'$ with $f(U)\subset U'$.
Then 
\begin{enumerate}[label=\roman*\emph{)},leftmargin=*]
\item
Functoriality: $  \SJ^{\ra k \la }(g\circ f) =   \SJ^{\ra k \la }g \circ  \SJ^{\ra k \la }f$  and $\SJ^{\ra k \la }\id_U =   \id_{\SJ^{\ra k \la }U} .$
\item
Homogeneity: for all $r \in \K^\times$, $  \SJ^{\ra k \la }f \circ \rho_{\la k \ra}(r)=\rho_{\la k \ra}(r) \circ  \SJ^{\ra k \la }f$.
\end{enumerate}
\end{theorem}

\begin{proof} Both statements 
 can be seen as a consequence of Theorems \ref{DifferenceFunctorTheorem} and 
 \ref{DifferenceImbThm} above.
An independent proof of  
(i) is given in \cite{Be10}, Theorem 2.10, and (ii) can also be proved by an easy
 direct computation.
\end{proof}

\section{Differential calculi}

Differential calculus is the extension of difference calculus to singular values.
In order to do this, we need additional structure, such as, e.g., topology.
We therefore
assume that $\K$ is a topological ring such that its unit group $\K^\times$ is
open dense in $\K$, and we assume
that $V,W$ are topological $\K$-modules, $U \subset V$ is open and
$f:U \to W$ is a continuous map. The class of continuous maps will be denoted by $\CC^0$
(see \cite{BGN04} for other ``$\CC^0$-concepts'').
There are two concepts of differential calculus, which we call ``cubic'' and ``simplicial''.

\begin{definition}[``cubic differentiability''] \label{cubicDef} We 
 say that $f:U \to W$ {\em is of class $\CC^{[1]}_\K$} (or just $\CC^{[1]}$ if the base ring
 $\K$ is clear from the context)  if there exists a continuous map
$f^{[1]}:U^{[1]} \to W$ extending the first order difference quotient map:
for all $(x,v,t) \in U^{]1[}$, we have
$f^{[1]}(x,v,t)=\frac{f(x+tv)-f(x)}{t}$.
By density of $\K^\times$ in $\K$, the map $f^{[1]}$ is unique if it exists, and so is
the value
$$
df(x)v:=f^{[1]}(x,v,0)=:\partial_v f(x).
$$
The {\em extended tangent map} is then defined by
$$
  T^{[1]} f: U^{[1]} \to W^{[1]},
 (x,v,t) \mapsto   T^{[1]}f(x,v,t)=:  T^{[1]}_{(t)}f(x,v):=\big( f(x), f^{[1]}(x,v,t),t \big) .
$$
The classes $\CC^{[k]}_\K$ (or shorter: just $\CC^{[k]}$) are defined by induction:
we say that $f$ is {\em of class $\CC^{[k+1]}$} if it is of class $\CC^{[k]}$
and if $f^{[k]}:U^{[k]} \to W$ is again of class $\CC^{[1]}$,
where $f^{[k]}:=(f^{[k-1]})^{[1]}$.
The {\em higher order extended tangent maps} are defined by
$  T^{[k+1]} f:=   T^{[1]} (  T^{[k]} f)$, and
the {\em $k$-th order cubic differentials}, at $x\in U$, are defined by 
$d^kf(x):V^k\to W, (v_1,\ldots,v_k)\mapsto \partial_{v_1}\ldots \partial_{v_k}f(x)$.
\end{definition}

\begin{theorem} \label{DifferentialFunctorTheorem}\label{SchwarzLemma}
Let $f:U \to W$ and $g:U' \to W'$ of class $\CC^{[k]}$ with $f(U)\subset U'$.
\begin{enumerate}[label=\roman*\emph{)},leftmargin=*]
\item
Functoriality: $  T^{[k]}(g\circ f) =   T^{[k]}g \circ  T^{[k]}f$ and $T^{[k]} \id_U =   \id_{T^{[k]}U}.$
\item
Homogeneity: for all $r \in \K^\times$, $  T^{[k]}f \circ \rho_{[k]}(r)=\rho_{[k]}(r) \circ   T^{[k]}f$.
\item
Linearity: the differential $df(x):V \to W$ is continuous and linear.
\item 
Symmetry: the $k$-th order cubic differential map $d^kf(x):V^k \to W$ is continuous,  $k$ times multilinear
and symmetric.
\end{enumerate}
\end{theorem} 

\begin{proof}
(i) and (ii) follow ``by density'' from Theorem \ref{DifferenceFunctorTheorem}.
Homogeneity of the differential is a special case of (ii), and additivity is proved in a
similar way (see \cite{BGN04}).
(iv) is a direct consequence of (iii) and of Schwarz' Lemma (\cite{BGN04}).
\end{proof}

Functoriality  is equivalent to saying that, for $t \in \K$ fixed,
$  T^{[1]}_{(t)}$ is a functor, and
for $t=0$ this gives the usual chain rule.
Moreover, for each $t$, the functor $  T^{[1]}_{(t)}$ commutes with direct products:
$ T^{[1]}_{(t)}(f \times g)$ is naturally identified with
$  T^{[1]}_{(t)} f \times   T^{[1]}_{(t)} g$. 
Analogously, for fixed time variables $\ttt$, $T^{[k]}_{(\ttt)}$ are functors commuting with direct products. 
In particular, for $\ttt=\ooo$,  $T^{[k]}_{(\ooo)}$ is a functor, called \textit{iterated tangent} functor  
and denoted by $T^k$. Note that $T^1=:T$ is the usual tangent functor. From this, 
we deduce that $T^{[k]}\K$, $T^{[k]}_{(\ttt)}\K$ and $T^k\K$ are again rings, with
 product and addition   obtained by applying the functor to product and addition in $\K$.
See \cite{Be10} for more information on these rings.

\begin{definition}[``simplicial differentiability''] \label{simplDef}
We say that $f$ {\em is of class $\CC^{\langle k \rangle}_\K$}, or just
{\em of class $\CC^{\langle k \rangle}$}, if, for all $1\leq \ell \leq k$,   there are
continuous maps $f^{\langle \ell \rangle}:U^{\langle \ell \rangle} \to W$
extending $f^{\rangle \ell \langle}$.
Note that, by density of $\K^\times$ in $\K$, the extension $f^{\langle \ell \rangle}$
is unique (if it exists), and hence in particular the value
$f^{\langle \ell \rangle}(\vvv ;\ooo)$,
called the {\em $\ell$-th order simplicial differential} is  uniquely determined.
We define also
$$
\SJ^{ \la k\ra} f: U^{\la k \ra } \to W^{\la k \ra }, \quad (\vvv;\sss) \mapsto
(f(v_0),f^{\la 1 \ra}(v_0,v_1;s_1),\ldots,f^{\la k \ra }(\vvv;\sss);\sss) ,
$$
and, for any fixed element $\sss \in \K^k$, we define the 
{\em simplicial $\sss$-extension of $f$} by
$$
\SJ_{(\sss)}^{ \la k\ra} f:  \SJ^{ \la k\ra}_{ ( \sss )} U \to W^{k+1}, \quad
\vvv \mapsto \SJ^{ \la k\ra}f(\vvv;\sss) \, ,
$$
where
$\SJ^{ \la k\ra}_{(\sss)} U := \bigl\{ \vvv \in V^{k+1} | \, (\vvv;\sss) \in U^{\la k \ra } \bigr\}$.
The {\em simplicial $k$-jet of $f$} is
$$
\SJ^k f := \SJ^{ \la k\ra}_{(\ooo)}f : \SJ^k U \to \SJ^k W, \quad
\vvv \mapsto \SJ^k f(\vvv) = \bigl( f^{\la \ell \ra}(v_0,\ldots,v_\ell;\ooo ) \bigr)_{\ell=0,\ldots,k} \, .
$$
where $f^{\la 0\ra}:=f$ by convention.
\end{definition}

\begin{theorem} \label{SimpDifferentialFunctorialityTheorem}
Let $f:U \to W$ and $g:U' \to W'$ of class $\CC^{\la k \ra }$ with $f(U)\subset U'$.
\begin{enumerate}[label=\roman*\emph{)},leftmargin=*]
\item
Functoriality: $  \SJ^{\la k \ra }(g\circ f) =   \SJ^{\la k \ra }g \circ  \SJ^{\la k \ra }f$  and 
$\SJ^{\la k \ra }\id_U =   \id_{\SJ^{\la k \ra }U} .$
\item
Homogeneity: for all $r \in \K^\times$, $  \SJ^{\la k \ra }f \circ \rho_{\la k \ra}(r)=\rho_{\la k \ra}(r) 
\circ  \SJ^{\la k \ra }f$.
\end{enumerate}
\end{theorem}

This follows ``by density'' from Theorem \ref{SimpDifferenceFunctorialityTheorem}.
As in the cubic case, for fixed time variables $\sss$, $\SJ^{\la k\ra}_{(\sss)}$
 is a functor commuting with direct products.  From this,
 it follows as above that $\SJ^{\la k\ra}\K$, $\SJ^{\la k\ra}_{(\ttt)}\K$ and $\SJ^k\K$ are again rings. 
Again, we refer to \cite{Be10} for more information on these rings. 
In particular, for $\sss=\ooo$, $\SJ^{\la k\ra}_{(\ooo)}$ is 
a functor called the \textit{$k$-th order jet functor}, denoted by $\SJ^k$. 
 Note that $\SJ^1=T^1$ is the usual tangent functor, and that
 for $\sss = \ooo$, functoriality gives
  equation (\ref{JetRule}).
According to \cite{Be10}, Theorem 1.7 and Corollary 1.11, there is an equivalent
characterization of the class $\CC^{\la k \ra }$ in terms of  ``limited expansions'',
having the advantage that no denominator terms appear:

\begin{theorem} \label{LimitedTheorem}
A map $f:U \to W$ is of class $\CC^{\langle k \rangle}$ if, and only if the following 
{\em simplicial limited expansions} hold: for all $1\leq \ell\leq k$, there exist continuous maps 
 $f^{\langle \ell \rangle}:U^{\langle \ell \rangle} \to W$ such that,
  whenever $(\vvv,\sss) \in U^{\langle \ell \rangle},$ 
\begin{eqnarray*}
f \bigl(v_0 + (s_1 - s_0)v_1) \bigr) & =
& f(v_0) + (s_1 - s_0) f^{\langle 1 \rangle} (v_0,v_1; s_0,s_1)
\cr
f \bigl(v_0 + (s_2 - s_0) (v_1 + (s_2 - s_1)v_2) \bigr)& =&
f(v_0) + (s_2 - s_0) f^{\langle 1 \rangle}(v_0,v_1;s_0,s_1)  + \cr
& & \quad 
(s_2 - s_1)(s_2-s_0) f^{\langle 2 \rangle} (v_0,v_1,v_2;s_0,s_1,s_2)
\cr
& \vdots & 
\cr
f \left(v_0 + \sum_{j=1}^k \prod_{\ell=0}^{j-1} (s_k - s_\ell) v_j\right) & =&
f(v_0) + \sum_{j=1}^k \prod_{\ell=0}^{j-1} (s_k - s_\ell) 
f^{\langle j \rangle} (v_0,\ldots,v_j;s_0,\ldots,s_j) 
\end{eqnarray*}
The maps $f^{\la\ell\ra}$ defined by this condition coincide with those defined in the
definition above.
\end{theorem}

\begin{theorem}[``cubic implies simplicial''] \label{CubictosimplicialTheorem}
If $f$ is of class $\CC^{[k]}$, then $f$ is of class $\CC^{\la k \ra }$, and the 
map $g_k:U^{\la k \ra } \to g_k(U^{\la k \ra }) \subset U^{[k]}$ defines a smooth imbedding of simplicial into cubic
differential calculus in the sense of Theorem
\ref{DifferenceImbThm}.
\end{theorem}

\begin{proof}
This follows ``by density'' from Theorem \ref{DifferenceImbThm}
(see \cite{Be10}, Theorem 1.6).
\end{proof}

We conjecture that also ``simplicial implies cubic'', i.e., 
{\em if $f$ is $\CC^{\la \infty \ra }$, then it is also
of class $\CC^{[\infty]}$}, but at present this conjecture is not settled. Therefore we will work
throughout with a $\CC^{[k]}$-assumption.

\begin{corollary}\label{SimplicialDiffSmooth} If $f$ is of class $\CC^{[k]}$, then $f^{\la j \ra}$ is of class $\CC^{[k-j]}$, for all $j\leq k$.
\end{corollary}

\begin{proof} 
Via the imbedding $g_j:U^{\la j \ra} \to U^{[j]}$,
we can consider $f^{\la j \ra}: U^{\la j \ra} \to W$ as a partial map 
of $f^{[j ]}: U^{[ j ]} \to W$. The latter is of class $\CC^{[k-j]}$, hence the former is
also of class $\CC^{[k-j]}$, and by the preceding theorem thus also of class
$\CC^{\la k-j\ra}$.
\end{proof}

\section{Continuous polynomial maps, and scalar extensions}
\label{PolyAppendix}

The purpose of this appendix is to show that continuous $\K$-polynomial mappings
are $\K$-smooth, and that their scalar extensions  by Weil $\K$-algebras $\bA$ are again
continuous, hence smooth over $\bA$.

\subsection{Continuous (multi-)homogeneous maps}
As in the main text, $V,W$ are topological modules over a topological ring $\K$.

\begin{theorem}[separation of homogeneous parts] \label{ContinuousDecomposition}
Assume $P = \sum_{i=0}^k P_i$ is a sum of $\K$-homogeneous maps $P_i:V \to W$
of degree $i$ (i.e., for all $r \in \K$, $P_i(r x)=r^i P_i(x)$).  Then the following are equivalent:
\begin{enumerate}
\item The map $P:V \to W$ is continuous.
\item For $i=0,\ldots,k$, the homogeneous part $P_i:V \to W$ is  continuous.
\end{enumerate}
Assume $P = \sum_{\Alpha\in \N^n} P_\Alpha :V^n\to W$ 
is a finite sum of  {\em multi-homogeneous} maps $P_\Alpha:V^n \to W$, 
with  $\Alpha=(\alpha_1,\ldots,\alpha_n)$, i.e., 
for all $r \in \K$ and  $1\leq i\leq n$, 
$P_\Alpha(x_1,\ldots, rx_i,\ldots, x_n)=r^{\alpha_i} P_\Alpha(x_1,\ldots,x_n)$.  
Then the following are equivalent:
\begin{enumerate}
\item The map $P:V^n \to W$ is continuous.
\item For all $\Alpha\in \N^n$, the multi-homogeneous part $P_\Alpha:V^n \to W$ is  continuous.
\end{enumerate}
\end{theorem}

\begin{proof}
We prove the first equivalence.
Obviously, (2) implies (1). Let us prove the converse. Assume that $P$ is continuous.
Since $P_0$ is constant, it is continuous. Without loss of generality,
we may assume that $P_0= 0$.
Fix scalars $r_1,\ldots,r_k \in \K$ and
define  continuous maps (depending on these scalars)
$$
Q_1(x):=r_1 P(x) - P(r_1x) = \sum_{i=2}^k (r_1 - r_1^i) P_i (x) ,
$$
$$
Q_{i+1}(x):= (r_{i+1})^{i+1} Q_i(x) - Q_i(r_{i+1} x).
$$
Then $Q_i$ is a linear combination of $P_{i+1},\ldots,P_k$, and in
particular, we find that
$
Q_{k-1}(x) = \lambda P_k(x)
$
with
$$
\lambda =
(r_1- r_1^k) (r_2^2 - r_2^k) \cdots (r_{k-1}^{k-1} - r_{k-1}^k) .
$$
In order to prove that $P_k$ is continuous,
it suffices to show that we
 may choose $r_1,\ldots,r_{k-1} \in \K$ such that the scalar $\lambda$ is
invertible, because then we have $P_k(x)=\lambda^{-1} Q_{k-1}(x)$.
Since
$Q_{k-1}$ is continuous by construction, it then follows that $P_k$ is continuous.

To prove our claim, write  $r_i^i - r_i^k = r_i^i (1 - r_i^m)$ with $m=k-i$.
Since the map $f:\K\to \K$, $r \mapsto 1-r^m$ is continuous and $\K^\times$
is open, $U:=f^{-1}(\K^\times)$ is open, and $U$ is non-empty since $0 \in U$.
Since $\K^\times$ is open and dense, $U':=U \cap \K^\times$ is open and
non-empty,
and we may choose $r_i \in U'$. Doing so for all $i$, we get an invertible
scalar
$\lambda$.

Having proved that $P_k$ is continuous, we replace $P$ by  $P-P_k$ and show as above
that $P_{k-1}$ is
continuous, and so on for all homogeneous parts.

\ssk
The second equivalence is proved similarly: proceed as above with respect
to the first variable $x_1$ in order to separate terms according to their
degree in $x_1$, then use the same procedure with respect to the
second variable $x_2$, and so on.
\end{proof}

\subsection{Continuous polynomial maps}
The following
 general definition of a $\K$-polynomial map, given in \cite{Bou}, ch.\ 4, par.\ 5, has been used in 
\cite{BGN04} (loc.\ cit., Appendix A, Def.\ A.5):

\begin{definition}\label{PolynomialDefinition}
A map $p:V \to W$ between $\K$-modules $V$ and $W$ is called {\em homogeneous
polynomial of degree $k$} if,
for any system $(e_i)_{i\in I}$ of generators of $V$, there exist coefficients
$a_\alpha \in W$ (where $\alpha: I \to \N$ has finite support and $\vert \alpha \vert := \sum_i \alpha_i = k$)
such that 
$$
p \left( \sum_{i\in I} t_i e_i \right) = \sum_{\alpha} t^\alpha a_\alpha 
\qquad  \mbox{ where } \, t^\alpha := \prod_i t_i^{\alpha_i} \, .
$$
If $V$ is free (in particular, if $\K$ is a field), then this is equivalent to saying that 
there exists a $k$ times multilinear map $m:V^k \to W$
such that $p(x)=m(x,\ldots,x)$.
In any case, a  {\em polynomial map} is a finite sum of homogeneous polynomial maps.

\ssk
The set of polynomial maps $p:V\to W$ is denoted by $\Pol(V,W)$, 
and we let $\Pol(V,W)_0 := \{ p:V \to W\ \textnormal{polynomial} \mid \, p(0)=0 \}$.
If $V=W$, the set of 
 polynomial maps $p:V\to V$ having an inverse polynomial map $q:V \to V$ 
forms  a group denoted by $\GPol(V)$, called the
{\em  general polynomial group of $V$}.  
By $\GPol(V)_0$ we denote the stabilizer subgroup of $0$. 
\end{definition}

Note that, if $p(x)=m(x,\ldots,x)$, then continuity of $p$ does not always imply
continuity of $m$ (if the integers are invertible in $\K$, then, by polarization, 
we may find symmetric and continuous $m$). 

\begin{theorem}\label{PolContinousSmooth}
Every continuous $\K$-polynomial map $p:V \to W$ is $\K$-smooth.
\end{theorem}

\begin{proof} Assume $p:V \to W$ is continuous polynomial.
By theorem \ref{ContinuousDecomposition},
we may assume without loss of generality  that $p$ is homogeneous of degree $k$. 
Assume first that $p(x)=m(x,\ldots,x)$ with multilinear $m:V^k \to W$.
Since $f$ is continuous, 
$$
F: V \times V  \to W, \quad
(x,v) \mapsto p(x+v) = m (x+v,\ldots,x+v) \, .
$$
is continuous and polynomial. 
Using multilinearity, we expand  
\begin{equation} \label{DecompF}
p(x+v)=m(x+v,\ldots,x+v)=\sum_{i=0}^k  M^{k-i,i}(x,v) ,
\end{equation}
where $M^{k-i,i}(x,v)$ is the sum of all terms in the expansion of $m$ containing
$i$ times the argument $v$ and $k-i$ times the argument $x$.  
The $i$-th term  in (\ref{DecompF}) is the homogeneous part of bi-degree $(k-i,i)$ 
of the continuous map $F$, and hence,  by  theorem \ref{ContinuousDecomposition},
 is again a continuous function of $(x,v)$.
Now, for $t \in \K^\times$,
we get by a similar expansion as above
$$
p^{[1]}(x,v,t) = \frac{f(x+tv)-f(x)}{t} =\sum_{i=1}^k t^{i-1} M^{k-i,i} (x,v) ,
$$
and since $M^{k-i,i}(x,v)$ is continuous, the right hand side defines a continuous
function of $(x,v,t)$, proving that a continuous extension of the difference quotient
function exists, and hence $p$ is $\CC^{[1]}$.
Moreover, $p^{[1]}$ is again continuous polynomial (of total degree at most $2k-1$), 
hence by induction it follows that
$p$ is of class $\CC^{[\infty]}$.

\ssk
If $p(x)$ is not directly given by a multilinear map $m$, then
 fix a system of generators $(e_i)$ of $V$ and consider the free module $E$ spanned
by the $e_i$, together with its canonical surjection $E \to V$.
The map $F$ lifts to map $\tilde F : E^2 \times \K \to W$, which we may decompose as
above, giving rise to a map $\tilde m$ and to maps $\tilde M^{k-i,i}$. Passing again to the quotient, we see that
the bi-homogeneous components $M^{k-i,i}$ are still continuous maps, and $p^{[1]}$
can be extended continuously as above.
\end{proof}

\begin{definition}
Assume $p:V \to W$ is a polynomial map of the form
$p(x)=m(x,\ldots,x)$ where $m:V^k \to W$ is $\K$-multilinear.
For $n \in \N$, $\Alpha =(\alpha_1,\ldots,\alpha_n) \in \N^n$ with
$\vert \Alpha \vert := \sum_{i=1}^n \alpha_i = k$ and $\vvv = (v_1,\ldots,v_n) \in V^n$ , let
$$
M^\Alpha (\vvv):= m(v_1:\alpha_1 \, ; \ldots; \, v_n:\alpha_n) 
$$
be the sum of all terms $m(w_1,\ldots,w_k)$ with exactly $\alpha_i$ among $w_1,\ldots,w_k$  equal to $v_i$.
\end{definition}

\begin{lemma}\label{MalphaLemma}
Assume  $p:V \to W$ is a polynomial map of the form
$p(x)=m(x,\ldots,x)$ where $m:V^k \to W$ is $\K$-multilinear.
Then, if $p$ is continuous, so is the  map
$$
M^\Alpha : V^n \to W, \quad \vvv = (v_1,\ldots,v_n) \mapsto M^\Alpha(\vvv) \, ,
$$
and it does not depend on the choice of $m$, so that we may write $p^\Alpha (\vvv):=M^\Alpha(\vvv)$. 
\end{lemma}

\begin{proof} This follows as above, by considering the continuous map
$$
F: V^n \to W, \quad \vvv  \mapsto p \left( \sum_{i=1}^n v_i \right) =
m \left( \sum_{i=1}^n  v_i ,\ldots,  \sum_{i=1}^n v_i  \right) 
$$
whose $\Alpha$-homogeneous component is precisely $M^\Alpha$.
\end{proof}

\subsection{Scalar extensions}
A {\em scalar extension} of $\K$ is, by definition, a commutative and associative $\K$-algebra
$\bA$. 
If $\bA$ is unital, then there is a natural map $\K \to \bA$, $t \mapsto t \cdot 1$. 

\begin{definition}\label{PolynomialScalarExtension}
Let $p:V \to W$ be a $\K$-polynomial map, homogeneous of degree $k$.
If  $p(x)=m(x,\ldots,x)$ with multilinear $m:V^k \to W$,
then the {\em scalar extension from $\K$ to $\bA$} of $p$ is the map 
$$
p_\bA : V_\bA = V \otimes_\K \bA \to W_\bA, \quad
v \otimes a \mapsto P(v) \otimes a^k = m_\bA (av,\ldots,av),
$$
where $m_\bA : (V_\bA)^k \to W_\bA$ is the multilinear map defined by the 
universal propery of the tensor product. 
If $V$ is not free, let as above $E \to V$ be the surjection defined by a system of
generators, define $\tilde P_\bA : E_\bA \to W_\bA$; then this map passes to
$V$ as a map $P_\bA:V_\bA \to W_\bA$.
\end{definition}

\begin{theorem}\label{PolynomialScalarExtensionSmooth}
Assume $\K$ is a topological ring with dense unit group, and $\bA$ a scalar extension of $\K$ 
which is a topological $\K$-algebra, homeomorphic to $\K^n$ with respect to some
  $\K$-basis $a_1,\ldots,a_n$.  
Assume $p:V \to W$ is continuous polynomial over $\K$. 
We equip $V_\bA$ with the product topology with respect to the decomposition
$$
V_\bA = V \otimes_\K (\bigoplus_{i=1}^n \K a_i) =
\bigoplus_{i=1}^n (V \otimes a_i) ,
$$
and similarly for $W_\bA$. 
Then $p_\bA : V_\bA \to W_\bA$ is a continuous $\bA$-polynomial map. 
\end{theorem} 

\begin{proof}
Note that the  topology on $V_\bA$ does not depend on the choice of the $\K$-basis of $\bA$ since the
group $\GL_n(\K)$ acts by homeomorphisms. 
Assume first that $p$ is homogeneous of degree $k$ and of the form
$p(x)=m(x,\ldots,x)$.  Then we have: 
\begin{eqnarray*}
p_\bA\left(\sum_{i=1}^n v_i\otimes a_i\right) &=&
m_\bA \left( \sum_{i=1}^n v_i\otimes a_i,\ldots, \sum_{i=1}^n v_i\otimes a_i\right) 
\\
&=&\sum_{j_1,\ldots,j_k=1}^n  m(v_{j_1},\ldots, v_{j_k})\otimes a_{j_1}\cdots a_{j_k}
\\
&=&\sum_{(\alpha_1\ldots,\alpha_n) \in \N^n , 
\sum_{i=1}^n \alpha_i=k} M^{\alpha_1,\ldots,\alpha_n}(v_1,\ldots,v_n)\otimes a_1^{\alpha_1}\cdots a_n^{\alpha_n},
\\
&=& \sum_{\Alpha \in \N^n, |\Alpha |=k}M^{\Alpha}(\vvv)\otimes \boldsymbol{a^\alpha} 
\end{eqnarray*}
According to Lemma \ref{MalphaLemma}, the map $M^\Alpha$ is  continuous, and hence
$p_\bA$ is continuous.
If $p$ is not of the form $p(x)=m(x,\ldots,x)$, then we 
use similar arguments as at the end of the proof of Theorem \ref{PolContinousSmooth}.
\end{proof}

\begin{remark}
If $\bA$ is unital, then we have a natural injection
$\sigma^\bA:V \to V^\bA$, $v \mapsto v \otimes_\K 1_\bA$, and 
$p_\bA$
``extends''  $p$ in the sense that $p_\bA\circ\sigma^\bA=\sigma^\bA\circ p$.

\end{remark}

\begin{remark}\label{ContinuousRemark}
If $p$ is as in the theorem, depending moreover continuously on a parameter $y$ (say, $p(x)=p_y(x)$, jointly continuous
in $(x,y)$), then, since $M^\Alpha$ does not depend on the choice of $m$ (see lemma \ref{MalphaLemma}),
the proof of the theorem shows that $(y,z) \mapsto (p_y)_\bA(z)$ is again jointly continuous in $(y,z)$.
\end{remark}


\end{document}